\documentclass[11pt, reqno]{amsart}
\usepackage{amscd,mathrsfs}
\usepackage{amsmath}
\usepackage{amssymb}
\usepackage{latexsym}
\usepackage{amsthm}
\usepackage{graphicx}
\usepackage{pb-diagram}
\usepackage[all]{xy}       % Include XY-pic

    \SelectTips{cm}{10}     % Use the nicer arrowheads

    \everyxy={<2.5em,0em>:} % Sets the scale I like

    \xyoption{web}          % Include the lattice feature, I don't know why%

\usepackage{hyperref}

\newcommand{\arxiv}[1]{\href{http://arxiv.org/abs/#1}{\tt arXiv:\nolinkurl{#1}}}
\newcommand{\arXiv}[1]{\href{http://arxiv.org/abs/#1}{\tt arXiv:\nolinkurl{#1}}}

\newtheorem{theorem}{Theorem}[section]
\newtheorem{lemma}[theorem]{Lemma}

\newtheorem{definition}[theorem]{Definition}
\newtheorem{example}[theorem]{Example}

\newtheorem{proposition}[theorem]{Proposition}
\newtheorem{corollary}[theorem]{Corollary}

\theoremstyle{remark}
\newtheorem{remark}[theorem]{Remark}

\numberwithin{equation}{section}

\newcommand{\nc}{\newcommand}

\nc{\nconv}{\mathop{\mbox{\large $\odot$}}}
\nc{\snconv}{\mbox{\scriptsize$\odot$}}
\nc{\flags}{\mathcal{F}}
\nc{\K}{\mathcal{K}}
\nc{\T}{\mathcal{T}}
\nc{\KP}{\operatorname{KP}}

\def\aa{{\bf a}}

\def\ii{{\bf i}}
\nc{\re}{re}

\def\N{\mathbb{N}}
\def\Q{\mathbb{Q}}

\def\C{\mathcal{C}}
\def\R{\mathbb{R}}
\def\Z{\mathbb{Z}}
\def\kk{\mathbf{k}}
\def\Cx{\mathbb{C}}
\def\E{E}
\def\I{\mathcal{I}}

\def\B{\mathcal{B}}

\def\w{\omega}

\def\uj{\bf j}

\def\L{\mathcal{L}}
\def\A{\mathcal{A}}

\DeclareMathOperator{\hit}{ht}

\def\ex{{\bf{ex}}}

\def\n{\mathfrak {n}}

\nc{\co}{\nabla}

\newcommand{\myrightleftarrows}{\mathrel{\substack{\longrightarrow \\[-.6ex] \longleftarrow}}}

\def\a{\alpha}
\def\b{\beta}
\def\e{\epsilon}

\def\la{\lambda}
\def\La{\Lambda}
\def\ga{\gamma}

\def\d{\delta}
\def\de{\delta}
\def\th{\theta}

\def\f{\mathbf{f}}\def\k{\mathbf{k}}
\def\g{\mathfrak{g}}
\def\B{\mathbf{B}}

\def\M{\mathcal{M}}
\def\D{\mathbb{D}}
\def\U{U}

\def\J{\mathcal{J}}

\DeclareFontFamily{U}{wncy}{}
    \DeclareFontShape{U}{wncy}{m}{n}{<->wncyr10}{}
    \DeclareSymbolFont{mcy}{U}{wncy}{m}{n}
    \DeclareMathSymbol{\shuffle}{\mathord}{mcy}{"58}

\def\Hom{\operatorname{Hom}}

\def\Ind{\operatorname{Ind}}

\def\seq{\,\mbox{Seq}\,}

\def\End{\operatorname{End}}
\def\Seq{\,\mbox{Seq}\,}

\def\Res{\operatorname{Res}}

\def\id{\mbox{id}}

\newcommand{\map}[2]{\,{:}\,#1\!\longrightarrow\!#2}

\def\h{\mathfrak{h}} % the upper half plane
\def\inv{^{-1}}

\newcommand{\sodot}{\mathop{\mbox{\normalsize$\bigodot$}}\limits}

\numberwithin{equation}{section}

\title[Cluster monomials are dual canonical]{Cluster monomials are dual canonical}
\address{University of Melbourne, Parkville, VIC 3010, Australia}\email{maths@petermc.net}
\author{Peter J McNamara}
\date{\today}

\begin{document}

\begin{abstract}
Kang, Kashiwara, Kim and Oh have proved that cluster monomials lie in the dual canonical basis, under a symmetric type assumption. This involves constructing a monoidal categorification of a quantum cluster algebra using representations of KLR algebras. We use a folding technique to generalise their results to all Lie types.
\end{abstract}
\maketitle

\tableofcontents

\section{Introduction} 

Let $G$ be a Kac-Moody group, $w$ an element of its Weyl group. Then associated to $w$ and $G$ there is a unipotent subgroup $N(w)$, whose Lie algebra is spanned by the positive roots $\a$ such that $w\a$ is negative.
Arguably the most important case is when $G$ is finite dimensional and $w$ is the longest element in the Weyl group, when $N(w)$ is maximal unipotent.
The coordinate ring $\Cx[N(w)]$ has a natural structure of a cluster algebra \cite{glsq1}. This story generalises to the quantum setting \cite{gls,my}, where the quantised coordinate ring $A_q(\mathfrak{n}(w))$ has a quantum cluster algebra structure. In this paper, we work with the quantum version. However, for this introduction, we shall continue with the classical story.

Kashiwara \cite{globalbases} and Lusztig \cite{lusztigoneofthem,lusztig91} defined a remarkable basis of the enveloping algebra $U(\mathfrak{n})$, called the lower global base or the canonical basis. It's dual, the dual canonical basis, is a basis of $\Cx[N]$ and is of concern to us. This dual canonical basis further induces a basis of the cluster algebra $\Cx[N(w)]$. 

It was a long-standing conjecture that the cluster monomials in the cluster algebra structure on $\Cx[N(w)]$ lie in the dual canonical basis. This was recently proved whenever $G$ is symmetric by Kang, Kashiwara, Kim and Oh \cite{kkkocombined}. Their proof used the categorification of $A_q(\mathfrak{n}(w))$ by categories of modules over Khovanov-Lauda-Rouquier algebras (henceforth called KLR algebras). They find a monoidal categorification of the cluster algebra structure inside these categories of KLR modules.

In this paper we follow in the footsteps of Lusztig's approach to the canonical basis in non-symmetric types using the technique of folding by an automorphism of the Dynkin diagram. The corresponding theory of folding KLR algebras was recently developed in \cite{folding}. Inside these folded categories, we are able to fold the monoidal categorification of \cite{kkkocombined} to deduce the fact that the cluster monomials lie in the dual canonical basis, and more generally, that they lie in the dual $p$-canonical basis for all primes $p$. The $p$-canonical basis is the analogue of the canonical basis defined using KLR algebras in characteristic $p$.

Our main theorem that we prove in this paper is the following: 

\begin{theorem}
 The algebra $\A_q(\n(w))$ has the structure of a (explicitly defined, independent of $p$) quantum cluster algebra in which every cluster monomial lies in the dual $p$-canonical basis.
\end{theorem}

Our main contribution is the categorification of the cluster algebra structure via folding, together with its implication about the cluster monomials belonging to the dual $p$-canonical basis. The fact that $\A_q(\n(w))$ has the structure of a quantum cluster algebra was established in this generality by Goodearl and Yakimov \cite{my,gy2}.

While this paper was being written, Qin \cite{qin2020} gave an independent proof of Theorem 1.1 for the dual canonical basis by different methods. That theorem is proved here as the special case $p=0$.

We now briefly discuss the contents of this paper. We begin with an overview of the theory of 
generalised minors and bases of canonical type (which are closely related to perfect bases).

We then summarise the necessary background results about folded KLR algebras. These results are all proved in \cite{folding}.

Finally we discuss the work of Kang, Kashiwara, Kim and Oh, and show how to incorporate the diagram automorphism into their story. This is where we prove the main results of the paper. These results comprise the categorification of the cluster algebra structure on $\mathbb{C}[N(w)]$ in terms of folded KLR algebras, and have Theorem \ref{main} as their most straightforward and classical corollary.

The results of this paper prove that certain modular decomposition numbers for KLR algebras are trivial. In particular, they prove that the reduction modulo $p$ of any irreducible module which corresponds to a cluster monomial remains irreducible.
Such results also have some geometric consequences, implying the non-existence of torsion in the stalks and costalks of the intersection cohomology of certain (Lusztig) quiver varieties. For example, by \cite[Theorem 3.7]{geordie}, there is no such torsion in $A_4$, as in this case $\A_q(\n)$ is a finite type cluster algebra.

We thank B. Leclerc for useful discussions about the theory of cluster algebras.

%%%%%%%%%%%%%%%%%%%%%%%%%%%%%%%%%%%%%%%%%%%%
\section{The quantum group}%%%%%%%%%%%%%%%%%%%
%%%%%%%%%%%%%%%%%%%%%%%%%%%%%%%%%%%%%%%%%%%%%%%%

Let $\g$ be a symmetrisable Kac-Moody Lie algebra and $I$ be a set indexing the simple roots.
To each $i\in I$ is an associated integer $d_i$ which is the entry of the symmetrising matrix.

We work over the ring $\Z[q,q\inv]$ or its fraction field $\Q(q)$. For each $i\in I$, let $q_i=q^{d_i}$. The quantum integer $[n]_i$ is $(q_i^n-q_i^{-n})/(q_i-q_i\inv)$. The quantum factorial $[n]_i!$ is the product $[n]_i[n-1]_i\cdots [1]_i$.

Let $U_q(\g)$ be the corresponding quantised enveloping algebra. Write $U_q(\n)$ for the upper-triangular part and let $\A_q(\n)$ be its graded dual. Let $\{\theta_i\}_{i\in I}$ denote the usual generating set of $U_q(\n)$ as an algebra, and use $\theta_i^*$ to denote the dual generating set of $\A_q(\n)$.
We work with the $\Z[q,q\inv]$-form of $U_q(\n)$ generated as an algebra by the divided powers $\th_i^{(n)}:=\th_i^n/[n]_i!$ and 
the corresponding $\Z[q,q\inv]$-form of $A_q(\n)$ which is its graded integral dual. When writing $U_q(\n)$ or $\A_q(\n)$, we will always refer to this integral form.

The bar involution on $\U_q(\n)$ is the automorphism fixing the Chevalley generators and sending $q$ to $q\inv$. This involution induces a bar involution on the dual $\A_q(\n)$. For homogeneous elements $a$ and $b$ of degrees $\a$ and $\b$ in $\A_q(\n)$, we have the formula \cite[Proposition 1]{leclerc1} \begin{equation}\label{barproduct}\overline{ab}=q^{\a\cdot\b}\bar{b}\bar{a}.
\end{equation}

% We define the linear endomorphisms $r_i$ and $_ir$ of $\A_q(\n)$ to be the adjoints of right and left multiplication by $\th_i^{(p)}$ respectively.
%Let $r_i\map{U_q(\n)}{U_q(\n)}$ be the linear map satisfying $r_i(\th_j)=\delta_{ij}$ and $r_i(xy)=q^{|y|\cdot i} r_i(x) y + x r_i(y)$.

Let $P$ be the weight lattice and let $P^+$ denote the set of dominant weights. If $i\in I$, let $\w_i$ be the corresponding fundamental weight. For each $\la\in P^+$, we denote by $V(\la)$ the irreducible highest weight $U_q(\g)$-module with highest weight $\la$. There is a partial order on $P$ where $\la \geq \mu$ if $\la-\mu$ is a sum of positive roots.

The algebra $\A_q(\n)$ is graded by $P^+$, we write $\A_q(\n)_\la$ for the $\la$'th graded piece. The graded components of the coproduct on $\A_q(\n)$ are denoted
\[
r_{\nu_1,\ldots,\nu_n}\map{\A_q(\n)_{\nu_1+\cdots+\nu_n}}{\A_q(\n)_{\nu_1}\otimes\cdots\otimes\A_q(\n)_{\nu_n}}.
\]

Write $E_i$ (the image of $\th_i$ under the inclusion of $U_q(\n)$ in $U_q(\g)$) and $F_i$ for the usual generators of $U_q(\g)$.
Let $\phi$ be the involutive antiautomorphism of $U_q(\g)$ sending $E_i$ to $F_i$. Again, we write $F_i^{(c)}:=F_i^c/[c]_i!$ for the divided power.

Let $(\cdot,\cdot)$ denote the $q$-Shapovalov form on $V(\la)$. This is a nondegenerate bilinear form such that
\[
 (xv,w)=(v,\phi(x)w)
\]
for all $x\in U_q(\g)$ and $v,w\in V(\la)$.
 We normalise the $q$-Shapovalov form so that $(v_\la,v_\la)=1$, where $v_\la$ is a chosen highest weight vector.

We now define a weight vector $v_\mu$ for all extremal weights $\mu$ in $V(\la)$.
Such an extremal weight $\mu$ is of the form $w\la$ for some $w$ in the Weyl group $W$. Let $w=s_{i_1}\cdots s_{i_n}$ be a reduced decomposition of $w$. Then we define
\[
 v_\mu = F_{i_1}^{(c_1)}\cdots F_{i_n}^{(c_n)} v_\la
\]
where the integers $c_k$ are defined by $c_k=(\a_{i_k},s_{i_{k+1}}\cdots s_{i_n}\la)$. The element $v_\mu$ does not depend on the choice of $w$ or the choice of reduced decomposition.

The $q$-Shapovalov form satisfies $(v_\mu,v_\mu)=1$ for all extremal weight vectors $v_\mu$. This is proved by a rank one computation.

The coproduct is denoted $r_{\nu_1,\ldots,\nu_n}\map{\A_q(\n)}{\A_q(\n)\otimes\cdots\A_q(\n)}$.

%%%%%%%%%%%%%%%%%%%%%%%%%%%%%%%%%%%%%%%%%%%%%%%%%%%%%%%%%%%%%%%%%%%%%%%%%
\section{Bases of canonical type}%%%%%%%%%%%%%%%%
%%%%%%%%%%%%%%%%%%%%%%%%%%%%%%%%%%%%%%%%%%%%%%%%%%%%%%%%%%%%%%%%%%%%%%%%%

We need the notion of a basis of dual canonical type. This is a strengthening of the notion of a perfect basis of $A_q(\n)$.

For $i\in I$ and $p\in \N$, define $_{i^p}r$ and $r_{i^p}$ to be the linear operators on $\A_q(\n)$ which are the adjoints of left and right multiplication by $\theta_i^{(p)}$ respectively.

Let $\sigma$ be the antiautomorphism which fixes the Chevalley generators. The bar involution is the automorphism fixing the Chevalley generators and sending $q$ to $q\inv$.

 \begin{definition}\label{def:can}
A basis $\mathbf B$ of $U_q(\n)$ is said to be of canonical type
if it satisfies the conditions~(\ref{it:BOCTa})--(\ref{it:BOTf}) below:
\begin{enumerate}
\item
\label{it:BOCTa}
The elements of $\mathbf B$ are weight vectors.
\item
\label{it:BOCTb}
$1\in\mathbf B$.
\item
\label{it:BOCTc}
Each right ideal $(\theta_i^p\mathbf U_q(\n)\otimes \Q(q))\cap U_q(\n)$ is spanned by a subset of
$\mathbf B$.
\item
\label{it:BOCTd}
In the bases induced by $\mathbf B$, the left multiplication by
$\theta_i^{(p)}$ from $U_q(\n)/\theta_i U_q(\n)$ onto
$\theta_i^pU_q(\n)/\theta_i^{p+1}U_q(\n)$ is given by a
permutation matrix.
\item
\label{it:BOCTe}
$\mathbf B$ is stable by $\sigma$.
\item
\label{it:BOTf}
$\mathbf B$ is stable under the bar involution.
\end{enumerate}
\end{definition}

% There is a dual notion, appropriately called a basis of dual canonical type.

\begin{definition}\label{defn:dualcantype}
A basis $\B^*$ of $\A_q(\n)$ is said to be of dual canonical type if it satisfies the conditions ~(\ref{it:BOdCTa})--(\ref{it:BOdCTf}) below:
\begin{enumerate}
\item
\label{it:BOdCTa}
The elements of $\mathbf B^*$ are weight vectors.
\item
\label{it:BOdCTb}
$1\in\mathbf B^*$.
\item
\label{it:BOdCTc}
Each $\ker(r_{i^p})$ is spanned by a subset of $\B^*$.
\item
\label{it:BOdCTd}
In the bases induced by $\mathbf B^*$, the map
\[
 r_{i^p}\map{\ker(r_{i^{p+1}})/\ker(r_{i^{p}})}{\ker(r_{i})}
\]
is given by a
permutation matrix.
\item
\label{it:BOdCTe}
$\mathbf B^*$ is stable by $\sigma$.
\item
\label{it:BOdCTf}
$\mathbf B^*$ is stable under the bar involution.
\end{enumerate}
\end{definition}

The dual basis to a basis of canonical type is of dual canonical type and vice versa.

It is proved in \cite[\S 5]{berensteinkazhdan} that each basis of dual canonical type induces a crystal structure on $\B$ (or $\B^\ast$) isomorphic to $B(\infty)$.

\begin{theorem}\label{basisofvla}\cite{baumann,folding}
 Let $\la\in P^+$. Let $V(\la)$ be a highest weight module of $U_q(\g)$ with highest weight $\la$ and let $v_\la$ be a highest weight vector. Let $\B$ be a basis of canonical type. Then the set
 \[
  \{bv_\la \mid b\in \B,bv_\la\neq 0\}
 \]
 is a basis of $V(\la)$.
\end{theorem}
% 
% \begin{proof}
%   (see \cite{baumann}). Let $v_\la$ be a highest weight vector. By \cite[Corollary 6.2.3(a)]{lusztigbook}, the map $x\mapsto x^-v_\la$ is a surjective map from $\f$ to $V(\la)$ with kernel
%   \[
%    K=\sum_{i\in I} \f \th_i^{\langle \a_i^\vee,\la\rangle +1}.
%   \]
% By condition (\ref{it:BOCTc}) in the definition of a bais of dual canonical type, $K$ is spanned by elements of $\B$. Hence $\B$ induces a basis of the quotient $V(\la)$.
% \end{proof}

The algebra $U_q(\n)$ is the quotient of the free $\Z[q,q\inv]$-algebra $'\f$ on the generators $\{\th_i\}_{i\in I}$. The shuffle algebra $\shuffle$ is defined as the graded dual of $'\f$, where the basis of words $[i_1,\ldots,i_n]$ in $\shuffle$ is dual to the basis of monomials in $'\f$. 
Then $\A_q(\n)$ is a subalgebra of $\shuffle$ and we write $\iota$ for the inclusion.

Let $i_1,i_2,\ldots$ be a sequence of elements of $I$ such that each element of $I$ appears infinitely often.
A word is said to be extremal for an element $x\in \A_q(\n)$ if it is minimal for the relevant lexicographical order amongst all words which appear in $\iota(x)$ with nonzero coefficient.

\begin{lemma}\label{extremalwords}
 Let $\B^*$ be a basis of dual canonical type and let $b^*\in\B^*$. Let $\ii=i_1^{a_1}i_2^{a_2}\cdots $ be an extremal word for $b^*$. Then $\ii$ appears in $\iota(b^*)$ with coefficient
\[
 [a_1]_{i_1}![a_2]_{i_2}!\cdots.
\]
\end{lemma}

\begin{proof}
Consider $_{i_1^{a_1}}r(b^*)$. By Condition (4) in Definition \ref{defn:dualcantype}, it lies in $\B^*$. The word $i_2^{a_2}i_3^{a_3}\ldots$ is extremal for $_{i_1^{a_1}}r(b^*)$ and by induction, we may assume that it appears with coefficient $[a_2]_{i_2}!\cdots$ in $\iota(_{i_1^{a_1}}r(b^*))$. This implies that $\ii$ appears in $\iota(b^*)$ with the desired coefficient, namely $[a_1]![a_2]!\cdots$.
\end{proof}

\begin{lemma}\label{product}\cite[\S 2.8]{kleshchev}
 Let $\B^*$ be a basis of dual canonical type and let $x$ and $y$ be two elements of $\B^*$. Expand their product in the basis $\B^*$:
\begin{equation}\label{exp}
 xy=\sum_{b^*\in\B^*}c_{b^*}b^*.
\end{equation}
Then there exists $b^*$ such that $c_{b^*}$ is a power of $q$.
\end{lemma}

\begin{proof}
  We have an extremal word $i_1^{a_1}i_2^{a_2}\ldots$ for $x$ and an extremal word $i_1^{b_1}i_2^{b_2}\cdots$ for $y$. Then the word $i_1^{a_1+b_1}i_2^{a_2+b_2}\ldots$ is extremal for $xy$ and appears with multiplicity
\[
 [a_1+b_1]![a_2+b_2]!\cdots
\]
by Lemma \ref{extremalwords}. Apply $\cdots r_{i_2^{a_2+b_2}}r_{i_1^{a_1+b_1}}$ to (\ref{exp}). The left hand side is a power of $q$, while by extremality, each term in the right hand side is $c_{b^*}$ times an element of $\B^*$ or zero. The element $1\in\B^*$ can only appear once, so there must be a term where $c_{b^*}$ is a power of $q$.
\end{proof}

For any weight basis $\B$, let $\B_\nu$ be the elements in $\B$ of weight $\nu$.

\begin{lemma}\label{gmlemma}
 Let $\B$ be a basis of canonical type. Let $\la\in P^+$ and let $\mu\leq \eta$ be two elements of $W\la$. Then there is exactly one choice of $b\in \B_{\eta-\mu}$ such that $b v_\mu\neq 0$. For this choice of $b$ we have $bv_\mu=v_\eta$.
\end{lemma}

\begin{proof}
Since the $q$-Shapovalov form is nondegenerate and the $\eta$-weight space of $V(\la)$ is one-dimensional, $bv_\mu\neq 0$ if and only if $(bv_\mu,v_\eta)\neq 0$.
 
We concentrate on the first statement and proceed by induction on $\eta$. First consider the base case when $\eta=\la$. In this case we have $(bv_\mu,v_\eta)=(v_\mu,\sigma(b)v_\eta)$. The invariance of $\B$ under $\sigma$ implies that $\phi(b)\in \B$ if and only if $b\in \B$. By Theorem \ref{basisofvla}, there is a unique choice of $\phi(b)\in \B$ making this pairing nonzero, hence a unique choice of $b\in \B$.

Now consider the case of $\eta\neq \la$. Then there exists $i$ such that $\eta< s_i\eta$. Then $v_\eta=F_i^{(c)} v_{s_i\eta}$ for some integer $c$. We compute
\[
 (bv_\mu,v_\eta)=(bv_\mu,F_i^{(c)}v_{s_i\eta})=(E_i^{(c)}bv_\mu,v_{s_i\eta}).
\]
If $b\in E_i U_q(\n)$ then $E_i^{(c)}bv_\mu$ factors through the weight space $V(\la)_{\eta-\a_i}$ which is zero. If $b\notin E_i U_q(\n)$, then there exists a unique $b'=\tilde e_i^{c}b\in\B$ such that $E_i^{(c)}b-b'\in E_i^{c+1}U_q(\n)$. By a similar argument, $E_i^{(c)}bv_\eta=b'v_\eta$. Here $\tilde e_i$ is the crystal operator.

By the inductive hypothesis, there exists a unique choice of $b'\in\B$ such that $(b'v_\mu,v_{s_i\eta})=0$. As $\tilde{e_i}$ is injective, there is thus at most one choice of $b$ such that $(bv_\mu,v_\eta)\neq 0$. The existence of such a $b$ is obvious as the condition $\mu\leq\eta$ means that it is easy to write down a product $x$ of Chevalley generators in $U_q(\n)_{\eta-\mu}$ such that $xv_\mu\neq 0$.

At this point we have proved the existence of exactly one choice of $b\in\B_{\eta-\mu}$ such that $bv_\mu\neq 0$. We now aim to show that for this choice of $b$, we have $bv_\mu=v_\eta$. We induct on $\eta$, the base case where $\eta=\mu$ being trivial.

Suppose then that $\eta\neq \mu$. Then there exists $i$ such that $\mu\leq s_i\eta<\eta$. Then by the inductive hypothesis, there exists $b'\in \B$ such that $b'v_\mu=v_{s_i\eta}$. Then $E_i^{(c)}b'v_\mu=v_\eta$ for some integer $c$ and we do a similar argument to show that $(\tilde e_i^c b') v_\mu=v_\eta$, so by uniqueness of $b$, we're done.
\end{proof}

%%%%%%%%%%%%%%%%%%%%%%%%%%%%%%%%%%%%%%%%%%%%%%%%%%%%%%%%%%%%%%%%%%%
\section{Generalised minors}%&%%%%%%%%%%%%%%%%%%%%%%%%%%%%%%%%%%%%
%&&&&&&&&&&&&&&&&&&&&&&&&&&&&&&&&&&&&&&&&&&&&&&&&&&&&&&&&&&&&&&&&&&&

\begin{definition}
 Let $\la\in P^+$ and $\mu,\eta\in W\la$. The generalised minor $D(\mu,\eta)\in \A_q(\n)$ is defined by
 \[
  D(\mu,\eta)(x)=( x v_\mu, v_\eta )
 \]
for all $x\in U_q(\n)$.
\end{definition}

The agreement between this definition and the definition of \cite[\S 9.1]{kkkocombined} is discussed in \cite[\S 5]{glsq1}.

\begin{theorem}\label{genminors}
Let $\B^*$ be a basis of dual canonical type. All generalised minors $D(\mu,\eta)$ with $\mu\leq \eta$ lie in $\B^*$.
\end{theorem}

\begin{proof}
Suppose $b\in\B$ where $\B$ is the basis dual to $\B^*$. Then $D(\mu,\eta)(b)=(b v_\mu,v_\eta)$. By Lemma \ref{gmlemma}, there is a unique choice of $b\in\B$ such that this is nonzero, and for this particular choice of $b$, $bv_\mu=v_\eta$. Therefore $D(\mu,\eta)\in\B^*$.
\end{proof}

% \begin{proof}
%  Note that for each extremal weight $\mu$, we need to have chosen an appropriately scaled weight vector $v_\mu$, this can be done so that the $v_\mu$'s are all images of each other under the $T_i$'s.
%  
%  Recall the definition
%  \[
%   \langle \Delta(\la,w),x\rangle = (v_{w\la},xv_\la).
%  \]
% 
%  
%  If $b\in \B_{w\la-\la}$, then $bv_\la=0$ except for one choice of $b$. Therefore $\langle\Delta(\la;w),b\rangle=0$ except for one choice of $b$. Hence $\Delta(\la;w)$ is a scalar multiple of an element of $\B^*$.
%  
%  Write $w=s_i u$ where $\ell(u)<\ell(w)$. Suppose $u\la-w\la=a\a_i$. Then
%  \begin{align*}
%   \langle \Delta(\la,w),F_i^{(a)}y\rangle &= (v_{w\la},F_i^{(a)}yv_\la) \\
%   &=(E_i^{(a)}v_{w\la},yv_\la) \\
%   &=(v_{u\la},yv_\la).
%  \end{align*}
% 
%  The left hand side is $\langle r_{i^a}(\D(\la,w)),y\rangle$. Therefore $r_{i^a}(\D(\la,w))=\D(\la,u)$ which lies in $\B^*$ by induction on the length of $w$ and therefore this scalar is 1. 
% \end{proof}

\begin{lemma}\label{restrictgm}
Let $\la\in P^+$ and suppose that $\mu_1<\mu_2<\cdots<\mu_{n+1}$ are weights in $W\la$. Then
\[
 r_{\mu_1-\mu_2,\ldots,\mu_{n}-\mu_{n+1}}(D(\mu_1,\mu_{n+1})) = D(\mu_1,\mu_2)\otimes\cdots\otimes D(\mu_{n},\mu_{n+1}).
\]
\end{lemma}
\begin{proof}
Let $\B$ be a basis of canonical type. Suppose that $b_1,\ldots,b_{n}\in\B$ are such that 
\[
 (r_{\mu_1-\mu_2,\ldots,\mu_{n}-\mu_{n+1}}(D(\mu_1,\mu_{n+1})),b_1\otimes \cdots \otimes b_n)\neq 0.
\]
Then $(D(\mu_1,\mu_{n+1}),b_1\cdots b_{n})\neq 0$. By the definition of the generalised minor, this implies that
\[
 (v_{\mu_1},b_1\cdots b_{n} v_{\mu_{n+1}})\neq 0.
\]
By repeated application of Lemma \ref{gmlemma}, there is a unique choice of $b_1,\ldots,b_{n}\in \B$ such that this pairing is nonzero, and for this choice of $b_1,\ldots,b_n$, the value of the pairing is 1.

Furthermore, by looking at the degrees involved, we see that the identity
 $$b_i
v_{\mu_{i+1}}=v_{\mu_i}$$ holds for each $i$.

Therefore we have identified $r_{\mu_1-\mu_2,\ldots,\mu_{n}-\mu_{n+1}}(D(\mu_1,\mu_{n+1}))$ as a tensor product of elements of $\B^*$.
Furthermore the proof of Lemma \ref{gmlemma} identifies each factor as $D(\mu_i,\mu_{i+1})$.
\end{proof}

\begin{lemma}\label{squaregm}
Let $\la$ be a dominant weight and $\mu,\zeta\in W\la$. Then
 \[
D(\mu,\zeta)^2=q^{(\mu-\zeta,\mu-\zeta)/2}D(2\mu,2\zeta).
 \]
\end{lemma}

\begin{proof}
The submodule of $V(\la)\otimes V(\la)$ generated by $v_\la\otimes v_\la$ is isomorphic to $V(2\la)$. Now let us specialise to $q=1$. Then we can define a bilinear form on $V(\la)\otimes V(\la)$ by setting
\[
 (x_1\otimes x_2,y_1\otimes y_2)=(x_1,y_1)(x_2,y_2)
\]
and extending by linearity, where $(\cdot,\cdot)$ is the Shapovalov form. When restricted to $V(2\la)$ inside $V(\la)\otimes V(\la)$ this form satisfies
\[
 (xv,w)=(v,\phi(x)w)
\]
for all $x\in\g$, hence is the Shapovalov form on $V(2\la)$.

Note that $v_\mu\otimes v_\mu$ and $v_\zeta\otimes v_\zeta$ are normalised extremal weight vectors in $V(2\la)$. Thus at $q=1$, we have
\begin{align*}
 D(2\mu,2\zeta)(x) &= (x(v_\mu\otimes v_\mu),v_\zeta\otimes v_\zeta) \\
 &= (D(\mu,\zeta)\otimes D(\mu,\zeta))(\Delta(x)) \\
 &= D(\mu,\zeta)^2(x).
\end{align*}

We have thus proved this lemma when $q$ is specialised to 1. Now pick a basis $\B^*$ of dual canonical type such that the structure constants for multiplication all lie in $\N[q,q\inv]$. For example, by \cite[Theorem 12.7]{folding}, we could take a basis coming from simple representations of KLR algebras. From Theorem \ref{genminors}, $D(\mu,\zeta)$ and $D(2\mu,2\zeta)$ both lie in $\B^*$. So by considering the expansion of $D(\mu,\zeta)^2$ in the basis $\B^*$ and comparing with what we already know about the behaviour at $q=1$, the only option is that $D(\mu,\zeta)^2=q^N D(2\mu,2\zeta)$ for some integer $N$.
We can identify the integer $N$ from the identity (\ref{barproduct}).
\end{proof}
\section{KLR algebras with automorphism}%%%%%%%%%%%%%%%%%%%%%%%%%%%%%%
%%%%%%%%%%%%%%%%%%%%%%%%%%%%%%%%%%%%%%%%%%%%%%%%%%%%%%%%%%%%%%%%%%%%%%%%%

Let $Q$ be a quiver with vertex set $I$ (not the same as the set $I$ in $\S 2$) and let $a$ be a finite order automorphism of $Q$. We assume that there are no arrows between any pair of vertices in the same $a$-orbit. Let $n$ be the order of $a$.

To such a quiver with automorphism, let $J$ be the set of $a$-orbits on $I$. Define $\cdot\map{J\times J}{\Z}$ by $j\cdot j=2|j|$ and for $j\neq k$, $-j\cdot k$ is equal to the total number of edges in $Q$ between an element of $j$ and an element of $k$. This is a symmetrisable Cartan datum expressed using the formulation in \cite{lusztigbook}. It is known that every symmetrisable Cartan datum arises from such a construction. The relationship between $(Q,a)$ and the Kac-Moody Lie algebra $\g$ we work with is that $\g$ is the Kac-Moody Lie algebra with Cartan datum $(J,\cdot)$.

%For each $\nu\in \N I$, the KLR algebra $R(\nu)$ has the folowing presentation:

Define, for any $\nu=\sum_{i\in I}\nu_i i\in\N I$, $|\nu|=\sum_{i\in I} \nu_i$ and 
\[
 \Seq(\nu)=\{\ii=(\ii_1,\ldots,\ii_{|\nu|})\in I^{|\nu|}\mid\sum_{j=1}^{|\nu|} \ii_j=\nu\}.
\]
This is acted upon by the symmetric group $S_{|\nu|}$ in which the adjacent transposition $(i,i+1)$ is denoted $s_i$.

Define polynomials $Q_{i,j}(u,v)$ for $i,j\in I$ by
\[
 Q_{i,j}(u,v)=\begin{cases}
\prod_{i\to j}(u-v) \prod_{j\to i}(v-u)
               &\text{if } i\neq j \\
0 &\text{if } i=j %\\
 %(v-u)^{-i\cdot j} &\text{if } j\to i.
              \end{cases}
\] where the products are over the sets of edges in $Q$ from $i$ to $j$ and from $j$ to $i$, respectively.

Let $k$ be an algebraically closed field whose characteristic does not divide $n$.

\begin{definition}
The KLR algebra $R(\nu)$ is the associative $k$-algebra generated by elements $e_{\bf i}$, $y_j$, $\tau_k$ with ${\bf i}\in \seq(\nu)$, $1\leq j\leq |\nu|$ and $1\leq k< |\nu|$ subject to the relations
\begin{equation} \label{eq:KLR}
\begin{aligned}
& e_\ii e_{\uj} = \delta_{\ii, \uj} e_\ii, \ \
\sum_{\ii \in {\rm Seq}(\nu)}  e_\ii = 1, \\
& y_{k} y_{l} = y_{l} y_{k}, \ \ y_{k} e_\ii = e_\ii y_{k}, \\
& \tau_{l} e_\ii = e_{s_{l}\ii} \tau_{l}, \ \ \tau_{k} \phi_{l} =
\tau_{l} \phi_{k} \ \ \text{if} \ |k-l|>1, \\
& \tau_{k}^2 e_\ii = Q_{\ii_{k}, \ii_{k+1}} (y_{k}, y_{k+1})e_\ii, \\
& (\tau_{k} y_{l} - y_{s_k(l)} \tau_{k}) e_\ii = \begin{cases}
-e_\ii \ \ & \text{if} \ l=k, \ii_{k} = \ii_{k+1}, \\
e_\ii \ \ & \text{if} \ l=k+1, \ii_{k}=\ii_{k+1}, \\
0 \ \ & \text{otherwise},
\end{cases} \\[.5ex]
& (\tau_{k+1} \tau_{k} \tau_{k+1}-\tau_{k} \tau_{k+1} \tau_{k}) e_\ii\\
& =\begin{cases} \dfrac{Q_{\ii_{k}, \ii_{k+1}}(y_{k},
y_{k+1}) - Q_{\ii_{k}, \ii_{k+1}}(y_{k+2}, y_{k+1})}
{y_{k} - y_{k+2}}e_\ii\ \ & \text{if} \
\ii_{k} = \ii_{k+2}, \\
0 \ \ & \text{otherwise}.
\end{cases}
\end{aligned}
\end{equation}
\end{definition}

We remark that we have not used the most generic choice of polynomials $Q_{ij}(u,v)$ as in \cite{rouquier} to define these algebras. However it is important to us that we do use this choice, which implies that these algebras are isomorphic to certain Ext algebras \cite{vv,rouquier2,maksimau} on the moduli stack of representations of the quiver $Q$. 
We rely on some results from \cite{kkkocombined} which require this geometric interpretation of these algebras. This assumption also implies that the algebras $R(\nu)$ \emph{symmetric} in the sense of \cite{kkk}, giving us access to the theory of $R$-matrices for KLR algebras.

The algebras $R(\nu)$ are $\Z$-graded by setting $e_\ii$ to have degree zero, $y_i$ to have degree 2 and $\tau_i e_\ii$ to have degree $\ii_i\cdot\ii_{i+1}$. All $R(\nu)$-modules which we consider in this paper will be graded left modules.

% \begin{lemma}
% Let $A$ be an algebra with an action of the finite group $G$. Then there is an equivalence of categories between the category of $G$-equivariant $A$-modules and the category of modules for the smash product $A\# G$.
% \end{lemma}
% 
% \begin{proof}
% If $M$ is an $A$-module and $g\in G$, then $g^*M$ and $M$ have the same underlying vector space.
% Thus, a $G$-equivariant $A$-module is the datum of a vector space $V$, an algebra homomorphism $A\to \End_k(V)$ and an automorphism $\sigma_g\in \Aut_k(V)$ for each $g\in G$, satisfying exactly the relations as in the definition of the smash product. Thus we obtain our equivalence of categories.
% \end{proof}

The automorphism $a$ of $Q$ induces isomorphisms $R(\nu)\cong R(a\nu)$ for all $\nu$. In particular, when $a\nu=\nu$, it induces an automorphism of the algebra $R(\nu)$, which we will also denote by $a$.

Consider $\nu$ such that $a\nu=\nu$. We consider the category $\C_\nu$ of pairs $(M,\sigma)$ where $M$ is a representation of $R(\nu)$ and $\sigma\map{a^*M}{M}$ is an isomorphism such that
\begin{equation}\label{atonis1}
 \sigma \circ  a^*\sigma\circ\cdots \circ (a^*)^{n-1}\sigma = \mathrm{id}_M.
\end{equation}

 A morphism from $(M,\sigma)$ to $(N,\tau)$ is a $R(\nu)$-module map $f\map{M}{N}$ such that the following diagram commutes:
 \[
  \begin{CD}
   a^*M @>a^*f>> a^*N\\
   @VV\sigma V   @VV\tau V\\
   M @>f>> N
  \end{CD}
 \]
 
 Let $\L_\nu$ be the full subcategory of $\C_\nu$ whose objects are pairs $(M,\sigma)$ where $M$ is finite dimensional. In \cite{folding} it is shown that $\C_\nu$ and $\L_\nu$ are abelian categories.

%%%%%%%%%%%%%%%%%%%%%%%%%%%%%%%%%%%%%%%%%%%%%%%%%%%%%%%%%%%%%%%%%
 \section{The Grothendieck group construction}%%%%%%%%%%%%%%%%%%%%
%%%%%%%%%%%%%%%%%%%%%%%%%%%%%%%%%%%%%%%%%%%%%%%%%%%%%%%%%%%%%%%%%

% From now on we make the assumption that the field $k$ contains $n$ $n$-th roots of unity.
Let $\Z[\zeta_n]$ denote the ring of cyclotomic integers where $\zeta_n$ is a primitive $n$-th root of unity and fix a ring homomorphism $\Z[\zeta_n]\to k$.
% All our references to \cite{folding}, which initially assumes an algebraically closed field, will be valid over $k$ by \cite[\S ??]{folding}, which ensures that all results descend to all ground fields $k$ with $n$ $n$-th roots of unity.

An object $(A,\sigma)$ of $\C_\nu$ is said to be \emph{traceless} if there is a representation $M$ of $R(\nu)$, an integer $t\geq 2$ dividing $n$ such that $(a^*)^tM\cong M$, and an isomorphism
\[
 A\cong M\oplus a^*M \oplus \cdots \oplus (a^*)^{t-1}M
\]
under which $\sigma$ corresponds to an isomorphism carrying the summand $(a^*)^jM$ onto $(a^*)^jM$ for $1\leq j<t$ and the summand $(a^*)^tM$ onto $M$.

The group $K(\L_\nu)$ is defined to be the $\Z[\zeta_n]$-module generated by symbols $[(M,\sigma)]$ where $(M,\sigma)$ is an object of $\L_\nu$, subject to the relations
\begin{align*}
 [X] &= [X']+[X'']  \hphantom{0}\quad \text{if $0\to X'\to X\to X''\to 0$ is exact}& \\
 [(M,\zeta_n\sigma)]&=\zeta_n[(M,\sigma)] & \\
 [X] &= 0 \quad\hphantom{[X']+[X''']} \text{if $X$ is traceless                             {}}
\end{align*}

There is also a version of this construction for the category of finitely generated projective modules. It plays an important role in \cite{folding} but is not needed in this paper.

% $$\begin{array}{lcl}
% z & = & a \\
% & = & a \\
% f(x,y,z) & = & x + y + z
% \end{array} $$

There is an action of $q$ on $K(\L_\nu)$ by shifting the grading. Thus $K(\L_\nu)$ is naturally a module over the ring $\Z[\zeta_n,q,q\inv]$.

% Given two objects $(M,\sigma)$ and $(N,\tau)$ in $\C_\nu$, there is an induced automorphism of $\Hom_{R(\nu)}(M,N)$, namely
% \[
%  f\mapsto (a^*)\inv (\tau^{-1}\circ f \circ \sigma).
% \]
% We will call this automorphism $a_{\sigma\tau}$.
% 
% By \cite{folding}, there is a pairing $K(\P_\nu)\times K(\L_\nu) \to \Z[\zeta_n]$ given by
% \[
%  \langle [(M,\sigma)],[(N,\tau)]\rangle =\mathrm{tr}(a_{\sigma\tau},\Hom_{R(\nu)}(M,N))
% \]
% and extending by linearity.

%%%%%%%%%%%%%%%%%%%%%%%%%%%%%%%%%%%%%%%%%%%%%%%%%%%%% 
 \section{Induction, restriction and duality}%%%%%%%%%%%%%%%%%%%%%%
%%%%%%%%%%%%%%%%%%%%%%%%%%%%%%%%%%%%%%%%%%%%%
 
 Given $(M,\sigma)$ and $(N,\tau)$ in $\C_\la$ and $\C_\mu$ respectively, we can form the induced module
 \[
  M\circ N:= R(\la+\mu)\bigotimes_{R(\la)\otimes R(\mu)} M\otimes N.
 \]
 The isomorphisms $\sigma$ and $\tau$ induce an isomorphism $a^*M\circ a^*N\to M\circ N$. When precomposed with the natural isomorphism $a^*(M\circ N)\cong a^*M\circ a^*N$, we obtain an isomorphism
 \[
  \sigma\circ \tau: a^*(M\circ N)\to M\circ N.
 \]
The object
\[
 (M,\sigma)\circ (N,\tau):= (M\circ N,\sigma\circ \tau)
\] is an element of $\C_{\la+\mu}$.

This induction functor induces a product structure on the direct sum of Grothendieck groups
\[
 \quad \bigoplus_{\nu\in \N J} K(\L_\nu).
\]

For $\la,\mu\in \N J$, let $e_{\la\mu}$ be the image of the identity under the inclusion $R(\la)\otimes R(\mu)\to R(\la+\mu)$. Given a $R(\la+\mu)$-module $M$, its restriction is defined by
\[
 \Res_{\la,\mu} M := e_{\la\mu} M.
\] It is a $R(\la)\otimes R(\mu)$-module.

Since $e_{\la\mu}$ is invariant under $a$, there is a canonical isomorphism $a^*(\Res M)\cong \Res(a^*M)$. Thus we obtain a restriction functor from $\C_{\la+\mu}$ to $\C_{\la\sqcup\mu}$.
This restriction functor induces a coproduct structure on the same direct sum of Grothendieck groups
\[
 \quad \bigoplus_{\nu\in \N J} K(\L_\nu),
\]
the details of which can be found in \cite{folding}.
% \begin{proposition}
%  The restriction functor fron  $\C_{\la+\mu}$ to $\C_{\la\sqcup\mu}$ induces a coproduct on 
%  \[
%  \bigoplus_{\nu\in \N J} K(\P_\nu) \quad \text{and} \quad \bigoplus_{\nu\in \N J} K(\L_\nu).
% \]
% \end{proposition}
% 
% 
% 
% \begin{proof}
% $$ K(\L_\la)\otimes K(\L_\mu) \cong K(\L_{\la\sqcup \mu})$$ follows since all traceless irreps in the tensor product are tensor products of traceless irreps.
% \end{proof}

Let $\psi$ be the antiautomorphism of $R(\nu)$ which sends each of the generators $e_\ii$, $y_j$ and $\tau_k$ to themselves.

Let $M$ be a finite dimensional $R(\nu)$-module. Then its dual $\D(M):=\Hom_k(M,k)$ is also an $R(\nu)$-module by
\[
 r(\la)(m)=\la(\psi(r)m)
\]
for all $r\in R(\nu)$, $\la\in\D(M)$ and $m\in M$. This extends to a contravariant autoequivalence of $\L_\nu$ which we also denote $\D$, where $\D(L,\sigma)=(\D L,(\D\sigma)\inv)$.

An object $(M,\sigma)$ of $\L_\nu$ is said to be self-dual if there is an isomorphism $$(M,\sigma)\cong (\D M,(\D\sigma)\inv).$$

In \cite{folding}, a ring $\Z\subset A\subset \Z[\zeta_n+\zeta_n\inv]$ is defined. This ring is equal to $\Z$ whenever $n<5$ or $k$ has characteristic zero, and there are no known examples where $A\neq \Z$. 

Let $\kk^*$ be the $A[q,q\inv]$-submodule of \(
  \bigoplus_{\nu\in \N J} K(\L_\nu)
\) spanned by the self-dual simple modules. For each $j\in J$, there is a canonical self-dual simple object of $\L_j$ which we denote by $L(j)$.
It is unique if $n$ is odd and unique up to rescaling $\sigma$ by $\pm 1$ if $n$ is even. A canonical choice is made in \cite[\S 7]{folding}. The main theorem of \cite{folding} is:

% \begin{theorem}\cite{folding}\label{projiso}
% There is a unique $\Z[q,q\inv]$-linear grading preserving isomorphism $\ga\map{\f}{\kk}$ such that
% \begin{enumerate}
%  \item $\ga(\th_j^{(m)})=[P_j^{(m)}]$ for all $j\in J$ and $m\in \N$.
% \item Under the isomorphism $\ga$, the multiplication $\f_\la\otimes \f_\mu\to\f_{\la+\mu}$ corresponds to the product on $\kk$ induced by $\Ind_{\la,\mu}$.
% \item Under the the isomorphism $\ga$, the comultiplication $\f_{\la+\mu}\to \f_\la\otimes \f_\mu$ corresponds to the  coproduct on $\kk$ induced by $\Res_{\la,\mu}$.
% \item Under the isomorphism $\ga$, the bar involution on $\f$ corresponds to the anti-linear involutive automorphism on $\kk$ induced by the duality $\D$.
% \end{enumerate}
%  \end{theorem}

\begin{theorem}\cite[Theorem 6.2]{folding}\label{thm:folding}
 There is an $A[q,q\inv]$-linear grading preserving isomorphism $\ga^*\map{\kk^*}{\A_q(\n)}\otimes_\Z A$ such that
\begin{enumerate}
 \item $\ga^*([L(j)])=\theta_j^*$ for all $j\in J$.
\item Under the isomorphism $\ga^*$, the multiplication $\A_q(\n)_\la\otimes \A_q(\n)_\mu\to \A_q(\n)_{\la+\mu}$ corresponds to the product on $\kk^*$ induced by $\Ind_{\la,\mu}$.
\item Under the the isomorphism $\ga^*$, the comultiplication $\A_q(\n)_{\la+\mu}\to \A_q(\n)_\la\otimes \A_q(\n)_\mu$ corresponds to the  coproduct on $\kk^*$ induced by $\Res_{\la,\mu}$.
\item Under the isomorphism $\ga^*$, the bar involution on $\A_q(\n)$ corresponds to the anti-linear antiautomorphism on $\kk^*$ induced by the duality $\D$.
%\item The isomorphism $\ga^*$ is the graded dual of the isomorphism $\ga$ in Theorem \ref{projiso}
\end{enumerate}
\end{theorem}

%Let $\mathscr{B}$ be the set of objects $(L,\sigma)$ which are isomorphic to their duals and for which $L$ is a simple $R$-module.
It is shown in \cite{folding} that classes of self-dual simple objects give a basis of $\kk^*$, which by Theorem \ref{thm:folding} is transported to a basis of $\A_q(\n)$. We denote this basis by $\B^*$. It is also shown that $\B^*$ is a basis of dual canonical type. We call $\B^*$ the dual $p$-canonical basis, where $p$ is the characteristic of $k$.
When $p=0$, $\B^*$ is the usual dual canonical basis, also known as the upper global basis.

% 
% Consider an orbit $j$ of vertices of $Q$, consisting of vertices $i_1,i_2,\ldots,i_t$ where $a({i_l})=i_{l+1}$ with indices taken modulo $t$. There are no arrows in the quiver $Q$ between these vertices. Then define the following simple module in $\C_j$:
% $$ L(j)=(L_j,\sigma)$$
% is defined in \cite{folding}
% 
% In \cite{folding}, the object $q_j^{{n \choose 2}}L(j)^{\circ n}$ is shown to be a self-dual simple object of $\C_{nj}$.
% 
% 
% Let $M\diamond N=\hd(M\circ N)$
% 
% Let $M=(L,\sigma)$ be a self-dual object of $\L$ with $L$ simple. We define
% \begin{align*}
%  \tilde e_j M &= q_j^{1-\epsilon_j(M)}\soc\Hom_{R(j)}(L(j),\Res_{j,\nu-j}M) \\
%  \tilde e_j^*M &= q_j^{1-\epsilon_j^*(M)}\soc\Hom_{R(j)}(\Res_{\nu-j,j}M,L(j)) \\
%  \epsilon_j(M) &= \max\{n\in\N \mid \tilde e_j^nM\neq 0\} \\
%  \epsilon_j^*(M) &= \max\{n\in\N \mid (\tilde e_j^*)^nM\neq 0\} \\
%  \tilde f_j M &= q_j^{\epsilon_j(M)} L(j)\diamond M \\
%  \tilde f_j^* M &= q_j^{\epsilon_j(M)} M\diamond L(j) 
% \end{align*}
% 
% \begin{theorem}
%  The set of self-dual simple objects of $\cup_{\nu\in\N I}\L_\nu$, under the above operations, forms a crystal which is isomorphic to $B(\infty)$ if $n$ is odd and $B(\infty)\sqcup B(\infty)$ if $n$ is even.
% \end{theorem}

%%%%%%%%%%%%%%%%%%%%%%%%%%%%%%%%%%%%%%%%%%%%%%%%%%%%%%%%%%%
\section{Cuspidal modules}
%%%%%%%%%%%%%%%%%%%%%%%%%%%%%%%%%%%%%%%%%%%%%%%%%%%%%%%%%%%%

For now we work with the unfolded Dynkin diagram obtained by forgetting the orientation on $Q$. Let $W$ be the corresponding Weyl group, generated by $s_i$ for $i\in I$. Let $\Phi^+$ be the set of positive roots and $\Phi^-$ be the set of negative roots. Fix an element $w\in W$. Define
\[
 \Phi(w)=\{\a\in \Phi^+\mid w(\a)\in \Phi^-\}.
\]
The following important fact is standard

\begin{proposition}\label{prop:standard}
 Let $w=s_{i_1}\cdots s_{i_l}$ be a reduced expression of $w\in W$. For each $1\leq k\leq l$, let $\b_k=s_{i_1}\cdots s_{i_{k-1}}\a_{i_k}$. Then
 \[
  \Phi(w)=\{\b_1,\b_2,\ldots,\b_l\}.
 \]
\end{proposition}

\begin{definition} \label{def:convex}  \cite[Definition 1.8]{tingleywebster}
 A convex preorder is a pre-order $\succ$ on $\Phi^+$
  such that, 
  \begin{enumerate}
\item \label{cpo2} For any equivalence class $\mathscr{C}$, any $a \in
  \text{span}_{{\Bbb R}_{\geq 0}} \mathscr{C}$ and any non-zero $\,x
  \in \text{span}_{{\Bbb Z}_{\geq 0}} \{ \beta \in \Phi^+ \mid \beta
  \succ \mathscr{C}\}$, we have that 
$a + x \not \in  \text{span}_{{\Bbb Z}_{\geq 0}} \{ \beta \in \Phi^+ \mid \beta \preceq \mathscr{C}\}.$
\item \label{cpo3} For any equivalence class $\mathscr{C}$, any $a \in \text{span}_{{\Bbb R}_{\geq 0}} \mathscr{C}$ and any non-zero $\,x \in \text{span}_{{\Bbb Z}_{\geq 0}} \{ \beta \in \Phi^+ \mid \beta \prec \mathscr{C}\}$, we have that 
$a+ x \not \in  \text{span}_{{\Bbb Z}_{\geq 0}} \{ \beta \in \Phi^+ \mid \beta \succeq \mathscr{C}\}.$
\end{enumerate}
A convex order is a convex pre-order which is a total order on real roots. 
\end{definition}

\begin{example}\label{ex:convexorder}
Let $(V,\leq)$ be a totally ordered $\Q$-vector space. Let $h\map{\Q\Phi}{V}$ be an injective linear transformation. For two positive roots $\a$ and $\b$, say that $\a\prec \b$ if $h(\a)/\hit(\a)<h(\b)/\hit(\b)$ and $\a\preceq \b$ if $h(\a)/\hit(\a)\leq h(\b)/\hit(\b)$. 
This defines a convex order on $\Phi$.
\end{example}

In the above example we can take $V=\R$ with the usual ordering to get the existence of many convex orders.

\begin{lemma}\label{lem:wconvex}
Fix a reduced expression $w=s_{i_1}\cdots s_{i_l}$ and let $\b_i$ be the root defined in Proposition \ref{prop:standard}. Then there exists a convex order $\prec$ such that $\b_1\prec\b_2\prec\cdots \prec \b_n$ and for any $\a\in \Phi^+\setminus \Phi(w\inv)$, $\a\succ\b_l$.
\end{lemma}

\begin{proof}
 Choose a generic hyperplane $\h$ in $\R \Phi$ such that the roots in $\Phi(w)$ and the roots in $\Phi^+\setminus \Phi(w)$ are on opposite sides of $\h$. Choose a linear map $h\map{\R\Phi}{\R}$, injective on $\Q\Phi$, such that $\h=\ker h$ and $h(\Phi(w))\subset \R_{<0}$. From $h$, the construction in Example \ref{ex:convexorder} provides us with a convex order $\prec'$ on $\Phi^+$. Now define our desired convex order $\prec$ by
 \begin{itemize}
  \item $\b_1\prec\b_2\prec\cdots \prec \b_n$
  \item If $\a\in\Phi(w)$ and $\b\notin\Phi(w)$ then $\a\prec\b$
  \item If $\a,\b\notin \Phi(w)$ then $\a\prec\b$ if and only if $\a\prec' \b$.
 \end{itemize}
It is straightforward to check that this construction gives a convex order.
\end{proof}

\begin{definition}
 Let $\a\in \Phi^+$. An object $(M,\sigma)$ of $\C_\a$ is $\prec$-cuspidal if whenever $\Res_{\la,\mu}M\neq 0$, we have that $\la$ is a sum of roots less than or equal to $\a$ while $\mu$ is a sum of roots greater than or equal to $\a$ under $\prec$.
\end{definition}

\begin{remark}
 Elsewhere, in  \cite{mcn3,tingleywebster}, this notion is called semicuspidal. Since we will only care about the case where $\a$ is a real root, the distinction between cuspidal and semicuspidal is irrelevant.
\end{remark}

\begin{theorem}\cite{tingleywebster}
 Let $\a$ be a real root and $\prec$ a convex order. Then there exists a unique self-dual simple $\prec$-cuspidal $R(\a)$-module, denoted $L(\a)$.
\end{theorem}

% As usual, we normalise the grading shift of $L(\a)$ so that it is self-dual.

The following theorem is \cite[Proposition 7.4]{glsq1} together with \cite[Theorem 9.1]{mcn3}. But we 
will give a direct proof.

\begin{theorem}\label{cuspgm}
Let $s_{i_1}s_{i_2}\cdots s_{i_l}$ be a reduced expression for $w\in W$ and let $\prec$ be a convex order constructed from this reduced expression as in Lemma \ref{lem:wconvex}. For each $k$ with $1\leq k\leq l$, 
let $\b_k=s_{i_1}\cdots s_{i_{k-1}}\a_{i_k}$ and let $L(\b_k)$ be the corresponding cuspidal $R(\b_k)$-module. Then
\[
  [L(\b_k)]=D(s_{i_1}\cdots s_{i_{k}}\omega_{i_k},s_{i_1}\cdots s_{i_{k-1}}\omega_{i_k}),
 \]
 the identity taking place in $\A_q(\n)$, identified with the Grothendieck group via Theorem \ref{thm:folding}.
\end{theorem}

\begin{proof}
Write $D$ for $D(s_{i_1}\cdots s_{i_{k}}\omega_{i_k},s_{i_1}\cdots s_{i_{k-1}}\omega_{i_k})$. Note that $\b_k=s_{i_1}\cdots s_{i_{k-1}}\omega_{i_k}-s_{i_1}\cdots s_{i_{k}}\omega_{i_k}$.
By Theorem \ref{genminors}, $D\in \B^*$ so by \cite[Theorem 12.7]{folding}, $D$ is the class of a self-dual simple module of $R(\b_k)$. 
Consider the classification of semicuspidal modules in terms of semicuspidal decompositions from \cite{tingleywebster}.
Since there is a unique irreducible cuspidal representation of $R(\b_k)$, it suffices to show that
\[
 r_{\b_l,\b_k-\b_l}(D)=0
\]
for all $l<k$.

Thus it suffices to show that $s_{i_1}\cdots s_{i_{k}}\omega_{i_k}-\b_l$ is not a weight in $V(\omega_{i_k})$. Since $\b_l$ is a real root, it suffices to show
\[
 (s_{i_1}\cdots s_{i_{k}}\omega_{i_k},\b_l)\leq 0.
\]
By the Weyl group invariance of the pairing $(\cdot,\cdot)$ and the fact that $s_{i_l}(\a_{i_l})=-\a_{i_l}$, this is equivalent to
\[
 (\omega_{i_k},-s_{i_{k-1}}\cdots s_{i_{l+1}}\a_{i_l})\leq 0.
\]
Since $s_{i_{k-1}}\cdots s_{i_{l}}$ is a reduced expression, $s_{i_{k-1}}\cdots s_{i_{l+1}}\a_{i_l}$ is a positive root, which thus has a nonnegative pairing with $\omega_{i_k}$, completing the proof.
\end{proof}

% %%%%%%%%%%%%%%%%%%%%%%%%%%%%%%%%%%%%%%%%%%%%%%%%%%%%%%%%%%%
% \section{PBW elements}
% %%%%%%%%%%%%%%%%%%%%%%%%%%%%%%%%%%%%%%%%%%%%%%%%%%%%%%%%%%%%%%
% 
% Let $w\in W$. Choose a reduced decomposition $w=s_{i_1}\cdots s_{i_n}$. This induces an ordering on
% \[
%  \Phi(w) = \{ \a\in\Phi^+\mid w^{-1}\a\in\Phi^{-}\},
% \]
% namely $\Phi(w)=\{\a_1,\ldots,\a_n\}$ where
% \[
%  \a_k = s_{i_1}\cdots s_{i_{k-1}}\a_{i_k}.
% \]
% 
% \begin{lemma}
%  There exists a convex order $\prec$ such that $\a_1\prec\a_2\prec\cdots \prec \a_n$ and for any $\b\in \Phi^+\setminus \Phi(w\inv)$, $\b>\a_n$.
% \end{lemma}
% 
% \begin{proof} $\Phi(w\inv)$ is biconvex. Put any convex ordering on $\Phi^+\setminus \Phi(w\inv)$ (e.g. one coming from a linear functional or charge) and glue with the desired ordering on $\Phi(w\inv)$.
% \end{proof}
% 
% 
% For each $\b\in\Phi(w)$, the root vector $E_{\b}$ is defined in the following manner, which depends on the choice of reduced decomposition.
% \[
%  E_{\a_k}= T_{i_1}\cdots T_{i_{k-1}} E_{i_k}.
% \]
% Here the $T_i$ are Lusztig's algebra automorphisms.
% 
% The dual root vector $E_\b^*$ for $\b\in\Phi(w)$ is defined as
% \[
%  E_\b^* = (1-q^{\b\cdot\b})E_\b
% \]
% using the identification of $\f^*\otimes \Q(q)$ and $\f\otimes \Q(q)$.
% 

%%%%%%%%%%%%%%%%%%%%%%%%%%%%%%%%%%%%%%%%%%%%%%%%%%%%%
\section{The quantum unipotent ring and its categorification}\label{qur}
%%%%%%%%%%%%%%%%%%%%%%%%%%%%%%%%%%%%%%%%%%%%%%%%%%%%%%%%%%%%%%0

Fix $w\in W$. Choose a reduced expression for $w$. This determines an enumeration of $\Phi(w)$ by $\b_1,\ldots,\b_{l}$ and some dual PBW vectors $E_{\b_i}^*\in \A_q(\n)_{\b_i}$. We will not give the usual definition of these in terms of the braid group action here, but will note that by Theorem \ref{cuspgm} and \cite[Proposition 7.4]{gls}, $E_\b^*=[L(\b)]$, so it can be defined as a generalised minor if desired.

Let $\A_q(\n(w))$ be the $\Z[q,q\inv]$-subring of $\A_q(\n)$ generated by $E_{\b_1}^*,\ldots,\E_{\b_l}^*$. It is known that this does not depend on the choice of a reduced decomposition.

Let $\pi=(\pi_1,\ldots,\pi_l)$ be a sequence of $l$ natural numbers. Associated to $\pi$ is the element
\[
 E_\pi^*=(E_{\b_1}^*)^{\pi_1} \cdots (E_{\b_l}^*)^{\pi_l}.
\]
Then this collection ${\E_\pi^*}$ is a $\Z[q,q\inv]$-basis of $\A_q(\n(w))$.

\begin{definition}
 Let $\C_w(\nu)$ be the full subcategory of finite dimensional $R(\nu)$-modules such that whenever $\Res_{\la,\mu}M\neq 0$, $\la\in \N\Phi(w)$. Let $\C_w=\cup_{\nu}\C_w(\nu)$.
\end{definition}

Since the restriction functor is exact, $\C_w$ is closed under subquotients and extensions. This implies it is abelian.

\begin{theorem}
The Grothendieck group of the category $\C_\w$ satisfies
 \[
  \bigoplus_{\nu\in\N I} K_0(\C_w(\nu)) \cong \A_q(\n(w))
 \]
\end{theorem}

\begin{proof}
 The category $\C_w$ is closed under the 
 induction product by the Mackey filtration \cite[Proposition 2.18]{khovanovlauda}. It also contains the modules $L(\b_k)$ since they are cuspidal. Therefore we have
 \[
  \bigoplus_{\nu\in\N I} K_0(\C_w(\nu)) \supset
  \A_q(\n(w))
 \]
To complete the proof, it will suffice to show that the dual PBW basis vectors $E_{\pi}^*$ span $K_0(\C_w)$, since they form a basis of $\A_q(\n(w))$.

The dual PBW basis vector $E_{\pi}^*$ is the class of the proper standard module $\overline{\Delta}(\pi)=L(\b_1)^{\circ \pi_1}\circ \cdots\circ L(\b_l)^{\circ \pi_l}$. By \cite[Theorem 10.1(3)]{mcn3} (although that paper is about affine Cartan data, the argument works in all types), the change of basis matrix between the proper standard modules and the simple modules is unitriangular. Since the classes of the simple modules in $\C_w$ form a basis of its Grothendieck group, the same is thus true for the classes of the proper standard modules, completing the proof.
\end{proof}

\begin{remark}
The above also shows that our category $\C_w$ is the same as the one with the same name in \cite{kkkocombined}, since both are Serre subcategories.
\end{remark}

%%%%%%%%%%%%%%%%%%%%%%%%%%%%%%%%%%%%%%%%%%%%%%%%%
\section{$R$-matrices}
%%%%%%%%%%%%%%%%%%%%%%%%%%%%%%%%%%%%%%%%%%%%%%%%%%%%

Since the polynomials $Q_{i,j}(u,v)$ appearing in the definition of the KLR algebras are all polynomials in $u-v$, the construction of $R$-matrices for KLR algebras in \cite{kkk} applies.

Thus for every pair of modules $X$ and $Y$, there is a nonzero morphism:
\[
 r_{X,Y}\map{X}{Y}.
\]

Given any three modules $X$, $Y$, and $Z$, these $R$
-matrices satisfy the Yang-Baxter equation
\begin{equation}\label{ybe}
 (r_{Y,Z}\circ \id_X)(\id_Y\circ r_{X,Z})(r_{X,Y}\circ \id_Z)=(\id_Z\circ r_{X,Y})(r_{X,Z}\circ \id_Y)(\id_X\circ r_{Y,Z})
\end{equation}
as well as the identity
\begin{equation}\label{yb2}
 (r_{Y\circ X,Z})(r_{X,Y}\circ \id_Z)=(\id_Z\circ r_{X,Y})(r_{X\circ Y,Z}).
\end{equation}

The $R$-matrices also behave nicely with respect to the automorphism $a$, namely we have the identity
\[
r_{a^*X,a^*Y}=a^*r_{X,Y}
 \]

\begin{lemma}
Let $X$ and $Y$ be modules such that $X\circ Y$ is simple. Then $X\circ Y\cong Y\circ X$ with a pair of inverse isomorphisms being given by $r_{X,Y}$ and $r_{Y,X}$.
\end{lemma}

\begin{proof}
When $q$ is specialised to 1, the Grothendieck group becomes commutative. Therefore, at $q=1$, $[Y\circ X]=[Y][X]=[X][Y]=[X\circ Y]$. So $[Y\circ X]$ has the same class as a simple module at $q=1$, which corresponds to forgetting the grading. Therefore $Y\circ X$ is simple.

 The $R$-matrices $r_{X,Y}$ and $r_{Y,X}$ must be isomorphisms by virtue of being nonzero maps between simple modules. 
 By \cite[Lemma 1.3.1(vi)]{kkk}, the $R$-matrices with spectral parameters satisfy $R_{Y,X}R_{X,Y}(v\otimes w)=(z-w)^t v\otimes w$ for vectors $v$ and $w$ of highest degree in $X$ and $Y$ respectively. Since $r_{Y,X}r_{X,Y}=(z-w)^{-s}R_{Y,X}R_{X,Y}|_{z=w=0}$ and is an isomorphism, it must be that $s=t$ and $r_{Y,X}r_{X,Y}$ is the identity (since it is a scalar multiple of the identity and preserves $v\otimes w$).
 Therefore the $R$-matrices $r_{X,Y}$ and $r_{Y,X}$ are inverses of each other.
\end{proof}

For such $X$ and $Y$, there are integers $\La(X,Y)$ and $\La(Y,X)$ such that
the modules $q^{\La(X,Y)}X\circ Y$ and $q^{\La(Y,X)}Y\circ X$ are self-dual simple modules (this is denoted $\tilde{\La}$ in \cite{kkkocombined}). Now consider the diagram
\[
q^{\La(X,Y)}X\circ Y \mathrel{\substack{r_{X,Y}\\\myrightleftarrows\\r_{Y,X}}}  q^{\La(Y,X)} Y\circ X.
\]
 We define $X\odot Y$ to be the direct limit of this diagram, it is thus a self-dual irreducible module which is canonically isomorphic to both $X\circ Y$ and $Y\circ X$ up to grading shift.

Now given $(X,\sigma)\in \C_\la$ and $(Y,\tau)\in\C_\mu$ with $X$ and $Y$ as above, the following diagram commutes

\centerline{\xymatrix{
&  q^{\La(X,Y)}X\circ Y \ar@/^0.3pc/@[red][r]^{r_{X,Y}} 
& q^{\La(Y,X)}Y\circ X  \ar@/^0.3pc/@[red][l]^{r_{Y,X}} 
  \\
& q^{\La(X,Y)}a^*(X\circ Y) \ar@/^0.3pc/@[red][r]^{a^*r_{X,Y}}\ar[u]^{\sigma\circ\tau} 
  & q^{\La(Y,X)}a^*(Y\circ X)\ar@/^0.3pc/@[red][l]^{a^*r_{Y,X}}  \ar[u]_{\tau\circ\sigma}
}}

so there is a canonical element $(X\odot Y, \sigma\odot \tau)$ 
in $\C_{\la+\mu}$ which does not depend on the order of the factors up to canonical isomorphism.

Now suppose we have a family of self-dual simple modules $X_1,\ldots, X_n$ such that $X_i\circ X_j$ is simple for all $i,j$. Then we define $\sodot_{i=1}^n X_i$ inductively by
\[
 \sodot_{i=1}^n X_i = \left(\sodot_{i=1}^{n-1} X_i\right) \odot X_n.
\]
Via the $R$-matrices, using (\ref{ybe}) and (\ref{yb2}), this module is canonically isomorphic to a grading shift of $X_{\sigma(1)}\circ\cdots\circ X_{\sigma(n)}$ for any permutation $\sigma\in S_n$. 

If in addition there are  isomorphisms $\sigma_i\map{a^*X_i}{X_i}$ satisfying (\ref{atonis1}), then by iterating the case of two modules, we obtain an object 
\[
 \left(\sodot_{i=1}^n X_i,\sodot_{i=1}^n \sigma_i\right)
\]
which does not depend on the order of the factors up to canonical isomorphism.

 %%%%%%%%%%%%%%%%%%%%%%%%%%%%%%%%%%%%%%%%%%%%%%%%%%%%%%%%%%%%%%%%%%%%%%%%%%%%%%%%%%
 \section{Reduction modulo $p$}%%%%%%%%%%%%%
%%%%%%%%%%%%%%%%%%%%%%%%%%%%%%%%%%%%%%%%%%%%%%%%%%%%%%%%%%%%%%%%%%%%%%%%%%%%%%%%%%%

We shall need to make some arguments involving reduction modulo $p$. For this we need a $p$-modular system. Since we can work over any field with $n$ $n$-th roots of unity, it is easy to find such a system. For example, we can take $F=\Q(\zeta_n)$, $\mathcal{O}=\Z[\zeta_n]$ and let $\pi$ be a place of $F$ over $p$ (recall that we assume $p$ does not divide $n$). Then $(F_\pi,\mathcal{O}_\pi,\mathbb{F}_p[\zeta_n])$ is a $p$-modular system. The KLR algebras are known to be free over $\Z$ \cite{maksimau}, so the standard theory of reduction modulo $p$ can be applied.

The following diagram of isomorphisms commutes, where $d$ is the decomposition map and the subscripts on the $\kk^*$s and $\ga^*$s refer to the characteristic of the field $k$ used to define them.

\centerline{\xymatrix{
&  {\k^*_0}  \ar[rr]^d\ar[dr]_{\ga_0^*}
&  
  & \k_p^*\ar[dl]^{\ga_p^*}
 \\
&  
  & \A_q(\mathfrak{n})\otimes A
  & 
}}

A module $M$ is said to be \emph{real} if $M\circ M$ is irreducible.

\begin{proposition}\label{mexists}
 Let $\la\in P^+$ and $\mu\leq\eta$ two elements of $W\la$. Then there exists a real self-dual simple $R(\eta-\mu)$-module $M(\eta,\mu)$ satisfying
\[
 \ga^*([M(\mu,\eta)])=D(\mu,\eta).
\]
Furthermore, the reduction of $M(\mu,\eta)$ modulo any prime remains irreducible.
\end{proposition}

\begin{proof}
 By Theorem \ref{genminors}, $D(\mu,\eta)\in\B^*$. Therefore a self-dual simple $R(\eta-\mu)$-module $M(\eta,\mu)$ must exist satisfying $\ga^*([M(\mu,\eta)])=D(\mu,\eta)$. This module is real because by Lemma \ref{squaregm}, $D(\mu,\eta)^2\in q^\Z \B^*$. It remains irreducible when reduced modulo $p$ because the decomposition map commutes with the isomorphism to $\A_q(\n)$.
\end{proof}

There are fundamentally three places in \cite{kkkocombined} where the assumption that the ground field $k$ has characteristic zero is used. This is in their proof that all generalised minors lie in $\B^*$, as well as in their Theorem 10.3.1 and Proposition 10.3.3. In the rest of this section, we show how to prove these results in arbitrary characteristic, ensuring that all results in \cite{kkkocombined} are valid in arbitrary characteristic.
 
We know that all generalised minors $D(\mu,\zeta)$ lie in $\B^*$ by Theorem \ref{genminors} and \cite[Theorem 12.7]{folding}. The other two results which we need to prove appear as Theorem \ref{4.28} and Corollary \ref{4.30} below.

% Let $\la$ be a dominant weight and let $\mu,\zeta\in W\la$ satisfy $\mu\leq \zeta$. Then we have proved in Theorem \ref{} that $D(\mu,\zeta)\in \B^*$. Therefore there exists an irreducible self-dual module $M(\mu,\zeta)$ satisfying
% \[
%  [M(\mu,\zeta)]=D(\mu,\zeta).
% \]

Given any two modules $M$ and $N$, we define $M\diamond N$ to be the head of $M\circ N$.

First we generalise \cite[Theorem 10.3.1]{kkkocombined} to all characteristics. 

 \begin{theorem}\label{4.28}
  Let $\la\in P^+$ and $\mu_1,\mu_2,\mu_3\in W\la$ such that $\mu_1\preceq \mu_2\preceq\mu_3$.
  Then
  \[
   M(\mu_1,\mu_2)\diamond M(\mu_2,\mu_3)\cong M(\mu_1,\mu_3).
  \]
 \end{theorem}

 \begin{remark}
It is possible to prove this in all characteristics from the characteristic zero result by an argument using reduction modulo $p$. However we give a uniform proof.
\end{remark}

\begin{proof}
By Lemma \ref{restrictgm}, we get
  \[
   \Res_{\mu_2-\mu_1,\mu_3-\mu_2}M(\mu_1,\mu_3)\cong M(\mu_1,\mu_2)\otimes M(\mu_2,\mu_3).
  \]
By adjunction there is thus a nonzero morphism from $M(\mu_1,\mu_2)\circ M(\mu_2,\mu_3)$ to $M(\mu_1,\mu_3)$. Since $M(\mu_1,\mu_2)$ is real, \cite[Theorem 3.2]{kkko} implies that $M(\mu_1,\mu_2)\circ M(\mu_2,\mu_3)$ has an irreducible head, which must thus be $M(\mu_1,\mu_3)$.
 \end{proof}

 \begin{lemma}\label{reducemodp}
 Let $M$ and $N$ be modules for a characteristic zero KLR algebra and let ${M}_p$ and ${N_p}$ denote reductions of $M$ and $N$ modulo the prime $p$. Suppose that ${M}_p\circ {M}_p$ and $N_p$ are simple. Then
\[
 \deg r_{M,N}=\deg r_{{M_p},{N}_p}.
\]
\end{lemma}

 \begin{proof}
Since ${M_p}\circ {M_p}$ is simple, the same is true of $M\circ M$. It is immediate from \cite[Theorem 3.2]{kkko} that the spaces $\Hom(M\circ N,N\circ M)$ and $\Hom({M_p}\circ {N_p},{N_p}\circ {M_p})$ are one dimensional, spanned by $r_{M,N}$ and $r_{{M_p}, N_p}$ respectively. The morphism $r_{M,N}$ can be reduced modulo $p$ to give a nonzero morphism in $\Hom({M_p}\circ {N_p},{N_p}\circ {M_p})$ of the same degree. By the one-dimensionality of this homomorphism space, we have our desired result.
\end{proof}

As a Corollary, we are able to prove \cite[Proposition 10.3.3]{kkkocombined} in all characteristics:

\begin{corollary}\label{4.30}
 Let $x\in W$ and $i\in I$ be such that $xs_i>x$ in Bruhat order and $x\w_i\neq \w_i$. Let $X=M(xs\w_i,x\w_i)$ and $Y=M(x\w_i,\w_i)$. Then
\[
% \mathfrak{d}(M(xs\w_i,x\w_i),M(x\w_i,\w_i))=1
\deg r_{X,Y}+ \deg r_{Y,X}=2.
\]
\end{corollary}

\begin{proof}
By Proposition \ref{mexists} and Lemma \ref{reducemodp}, we reduce ourselves to the case where the ground field $k$ is of characteristic zero, which is \cite[Proposition 10.3.3]{kkkocombined}.
\end{proof}

As a consequence of these results, all results proved in \cite{kkkocombined} are valid over an arbitrary ground field.

%%%%%%%%%%%%%%%%%%%%%%%%%%%%%%%%%%%%%%%%%%%%%%%%%%%
\section{Quantum cluster algebras}\label{qcas}
%%%%%%%%%%%%%%%%%%%%%%%%%%%%%%%%%%%%%%%%%%%%%%%%%%%%

Here we give the definition of a skew-symmetrisable quantum cluster algebra, introduced by Berenstein and Zelevinsky \cite{quantumcluster}. Some of our powers of $q$ are different from those which appear elsewhere in the literature. We make these choices so that we do not have to ever extend scalars to $\Z[q^{ 1/2},q^{-1/2}]$. Specialising $q=1$ recovers the classical notion of a commutative cluster algebra.

% First we recall the basics of the construction of a quantum cluster algebra, from \cite{quantumcluster}.

Let ${\bf{ex}}\subset S$ be two finite sets. A cluster matrix is an integer matrix $B$ with rows labelled by $S$ and columns labelled by $\bf{ex}$. We call elements of $S$ vertices and elements of $\bf{ex}$ exchangeable vertices. Vertices not in $\bf{ex}$ are called frozen.

Let $\La=(\la_{ij})_{i,j\in S}$ be a skew-symmetric integer matrix. 
We say that the pair $(\La,B)$ is \emph{compatible} if 
\[
 \La B = -2E
\]
where $E=(e_{st})$ is a matrix with $e_{st}\geq 0$ for all $s,t$ and $e_{st}>0$ if and only if $s=t\in{\bf{ex}}$. We write $e_s$ for the integer $e_{ss}$. If $(\La,B)$ is compatible, then the principal part of $B$ is automatically skew-symmetrisable.

The mutation of $(\La,B)$ in the direction $k\in \ex$ is the pair $(\La',B')$ where
\[
 \la'_{st}=\begin{cases}
           \la_{st}&\text{if }s,t\neq k \\
           \sum_i \max(b_{ik},0)\la_{it} &\text{if }s=k\neq t,
          \end{cases}
\]
and
\[
 b'_{ij} = \begin{cases}
            -b_{ij} &\text{if } i=k \text{ or } j=k \\
            b_{ij}+\frac{|b_{ik}|b_{kj}+b_{ik}|b_{ij}|}{2}&\text{otherwise}.
           \end{cases}
\]
It is easily checked that the mutation of a compatible pair is compatible.

Let $Q$ be a lattice and $(\cdot,\cdot)\map{Q\times Q}{\Z}$ a bilinear form. Let $\A$ be a $Q$-graded $\Z[q,q\inv]$-algebra which embeds in a skew-field $\K$. Suppose $\A$ has a grading-preserving bar-involution which satisfies $\overline{q}=q\inv$ and
\[
 \overline{AB} = q^{(a,b)} \overline{B}\cdot\overline{A}
\]
for homogeneous elements $A$ and $B$ of degrees $a$ and $b$ respectively.

Fix a compatible pair $(\La,B)$ and let $\{d_i\}_{i\in S}$ be a family of elements in $Q^+$ such that $\la_{ij}\cong (d_i,d_j)\pmod 2$ for all $i,j\in S$. Let $Y_i\in \A_{d_i}$ for $i\in S$ be a family of elements such that
\[
 Y_i Y_j = q^{\la_{ij}} Y_j Y_i.
\]
We call such a family of elements $\La$-commuting.

Given such a choice of $\La$-commuting elements, make a choice of identification $S\cong\{1,2,\ldots,m\}$. Let $\aa=(a_1,\ldots,a_m)\in \N^m$. Define
\[
 Y^\aa = (q^{1/4})^{(\sum_i a_i d_i,\sum_i a_i d_i)-\sum_ia_i(d_i,d_i)+2\sum_{i>j}a_ia_j\la_{ij}}Y_1^{a_1}Y_2^{a_2}\cdots Y_m^{a_m}.
\]
This is an element of a $\A$ that does not depend on the ordering of the indexing set $S$. The condition $\la_{ij}\cong (d_i,d_j)\pmod 2$ implies that the exponent of $q$ is an integer. If each $Y_i$ is self-dual (i.e. invariant under the bar involution), then so is $Y^\aa$.

The based quantum torus associated with $\La$ is the $\Z[q,q\inv]$-algebra $\mathcal{T}(\La)$ generated by $X_1^{\pm 1},\ldots,X_n^{\pm 1}$ subject to the relations
\[
 X_i X_j = q^{\la_{ij}} X_j X_i.
\]

\begin{definition}
A quantum seed in $\A$ is a triple $(\{Y_s\}_{s\in S},\La,B)$ with $(\La,B)$ a compatible pair and $\{Y_s\}_{s\in S}$ a $\La$-commuting family of self-dual elements of $\A$ such that the induced map from $\mathcal{T}(\La)$ to $\K$ is injective.
\end{definition}

In such a situation, we call the set $\{Y_s\}_{s\in S}$ the cluster, and the $Y_s$ are the cluster variables. If $s\in{\bf{ex}}$ then the corresponding cluster variable is called exchangeable, otherwise frozen. The elements $Y^{\bf a}$ are the quantum cluster monomials corresponding to this seed.

We now define how to mutate a quantum seed at an exchangeable variable.
Fix $s\in \ex$.
Define $a^+(s)_t=\max({b_{ts},0})$ and $a^-(s)_t=\max(-b_{ts},0)$. This defines two sequences $\aa^+(s)$ and $\aa^+(s)$ of integers indexed by $S$.
The mutation variable $Y_s'$ is then defined by
\begin{equation}\label{12.n}
 Y_s {Y_s}' = Y^{\aa^+(s)}+ q^{e_s} Y^{\aa^-(s)}.
\end{equation}

It is easy to check the following
\begin{proposition}
The triple $(\{Y_t\}_{t\in S}\cup \{Y'_s\}\setminus \{Y_s\},\La',B')$ is also a quantum seed in $\K$.
\end{proposition}

Now we can define a quantum cluster algebra.

\begin{definition}
 Let $\I$ be a quantum seed in $\A$. The quantum cluster algebra $\A(\I)$ associated to $\I$ is the $\Z[q,q\inv]$-subalgebra of $\K$ generated by all quantum cluster variables in all quantum seeds obtained from $\I$ by all sequences of mutations.
\end{definition}

A quantum cluster algebra will necessarily be a subalgebra of the fraction field of $\mathcal{T}(\La)$.

\section{The initial quiver}\label{subsec}
%%%%%%%%%%%%%%%%%%%%%%%%%%%%%%%%%%%%%%%%%%%%%%%%%%%

The quivers that appear from now on will always refer to combinatorial data used in cluster mutation. They bear no relation to the quiver used to define the KLR algebras, which can be safely forgotten about.

We now give a construction of a quantum cluster algebra associated to an element $w\in W$.
The construction will a priori depend on a choice of reduced expression for $w$, but ultimately we will show that this quantum cluster algebra does not depend on this choice.

Let $w=s_{i_1}\ldots s_{i_m}$ be a reduced expression for $w\in W$. To this data, we now define a quiver $Q(i_1,\ldots,i_m)$. Consider an array indexed by $ \{1,\ldots,m\}\times I$. We place a vertex at each point in the array of the form $(t,i_t)$. We place a horizontal arrow from the vertex $(a,i)$ to the vertex $(b,i)$ if $a>b$ and there is no vertex $(c,i)$ with $a>c>b$. The other arrows between rows form a zigzag pattern: We place $-i\cdot j$ arrows from $(a,i)$ to $(b,j)$ if $a<b$ and there is no $(c,j)$ with the properties that (i) $c>b$ and (ii) there is no vertex $(d,i)$ with $c>d>b$. For each $i\in I$, the vertex of the form $(a,i)$ with largest $a$ is decreed to be frozen.

For example suppose that our Cartan matrix is
\[
 \begin{pmatrix}
  2 & -3 & -4 \\
  -3 & 2 & -2 \\
  -4 & -2 & 2
 \end{pmatrix}
\]
and that our reduced word is $s_is_js_is_ks_is_js_is_js_ks_j$, where $i$, $j$ and $k$ are used to label the rows of the Cartan matrix, in that order. Then the quiver $Q(i,j,i,k,i,j,i,j,k,j)$ is
$$
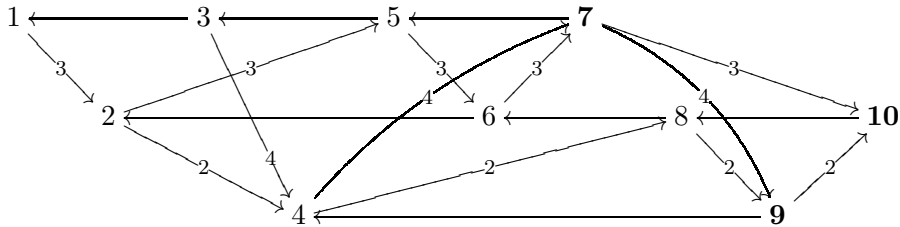
\xymatrix{
1\ar[dr]|3&&3\ar[ddr]|>>>>>>4\ar[ll]&&5\ar[dr]|3\ar[ll]&&{\bf 7}\ar[ll]\ar[drrr]|3 \ar@/^1pc/[rrdd]|4&& \\
&2\ar[drr]|2\ar[urrr]|3&&&&6\ar[ur]|3\ar[llll]&&8\ar[dr]|2\ar[ll]&&\ar[ll]{\bf 10}  \\
&&&4\ar[urrrr]|2\ar@/^0.8pc/[uurrr]|4&&&&&{\bf 9}\ar[ur]|2\ar[lllll]&
}
$$
% 
% $$
% \xymatrix{
% &&& {\bf 7} \ar@/_1pc/[ddll]|2 \ar[dlll]|3 \ar[rr] && 5 \ar[dl]|3 \ar[rr] && 
% 3 \ar[ddl]|>>>>>>>2 \ar[rr] && 1 \ar[dl]|3\\
% {\bf 10} \ar[rr] && 8 \ar[dl]|2\ar[rr] && 6 \ar[ul]|3\ar[rrrr] &&&& 
% 2 \ar[ulll]|3\ar[dll]|<<<<<<2 \\
% & {\bf 9} \ar[ul]|2\ar[rrrrr] &&&&& 4 \ar[ullll]|2\ar@/_0.8pc/[uulll]|2
% }
% $$
where a number on an arrow means that there are that many arrows between the vertices. Frozen vertices are depicted in bold.

To the vertex $(t,i_t)$, we associate the cluster variable
\begin{equation}\label{eq:defyt}
 Y_t:=D(s_{i_1}\cdots s_{i_t}\w_{i_t},\w_{i_t}).
\end{equation}

These are the cluster variables in an initial cluster in the quantum cluster algebra structure on $\A_q(\n(w))$.
%
%\subsection{The magic mutation sequence}.
%
%We now describe the mutation sequence of \cite[\S 13]{glsq1} (this is our name to describe their algorithm).
%
%This is a sequence of mutations of the quiver $Q(i_1,\ldots,i_m)$ constructed above, which consists of $m$ subsequences of mutations $\overrightarrow{\mu_1},\overrightarrow{\mu_2},\ldots,\overrightarrow{\mu_m}$. The sequence of mutations $\overrightarrow{\mu_t}$ consists only of mutations in the row indexed by $i_t$. Suppose that there are $a$ vertices in row $i$. Then the $k$-th mutation in row $i$ consists of mutating at the leftmost $a-k$ vertices in row $i$, in order from left to right.
%
%For example, continuing the example from \S\ref{subsec}, the magic mutation sequence is $$(2,6,2,1,4,3,1,8,6,2,5,3,1).$$
%
%
%
%\begin{lemma}
% During the application of the magic mutation sequence, while row $i$ is being mutated, there are no new edges in the quiver created that do not involve vertices in row $i$.
%\end{lemma}

%%%%%%%%%%%%%%%%%%%%%%%%%%%%%%%%%%%%%%%%%%%%%%%%%%%
\section{Seeds with automorphism}%%%%%%%%%%%%%%%%%%%%%%%%%%%%
%%%%%%%%%%%%%%%%%%%%%%%%%%%%%%%%%%%%%%%%%%%%%%%%%

\begin{definition}\label{def:qmswa}
A quantum monoidal seed with automorphism is a quadruple \newline $\J=(\{M_i\}_{i\in V},Q,a,\{\sigma_i\}_{i\in V})$ where 
\begin{enumerate}
 \item $Q$ is a quiver without loops with vertex set $V=V_{\rm ex}\sqcup V_{\rm fr}$.
\item $a$ is an automorphism of $Q$ preserving the decomposition $V=V_{\rm ex}\sqcup V_{\rm fr}$ such that there are no arrows in $Q$ between any two vertices in the same $a$-orbit.
\item $\{M_i\}_{i\in V}$ is a family of modules such that $M_i\circ M_j$ is simple for all $i,j\in V$.
 \item $\{\sigma_i\}_{i\in V}$ is a family of isomorphisms
\[
 \sigma_i\map{a^*M_i}{M_{ai}}
\] such that the composition
$$\sigma_{a^{n-1}i}\circ a^*\sigma_{a^{n-2}i} \circ \cdots \circ (a^*)^{n-1}\sigma_i$$
is the identity map on $(a^*)^nM_i = M_i = M_{a^n i}$.
\end{enumerate}
\end{definition}

We remark here an important consequence of condition (3) in this definition. It implies that any $M_{i_1}\circ\cdots \circ M_{i_n}$ is simple for any $i_1,\ldots,i_n\in V$. This follows from \cite[Prop 3.2.5]{kkkocombined}.

Let $s$ denote an orbit of $a$ on $V_{\rm ex}$. We define the \emph{mutation} $\mu_s(Q)$ of $Q$ in the direction $s$ to be the quiver obtained by the following combinatorial rule.

\begin{enumerate}
 \item For every $i\in s$ and pair of arrows $j\to i\to k$ in $Q$, we add an arrow $j\to k$.
\item Reverse the direction of any arrow involving a vertex in $s$.
\item If there is a pair of opposing edges $x\to y$ and $y\to x$ in the quiver, delete both. Repeat until no such opposing pair exists.
\end{enumerate}

It is clear that $a$ induces an automorphism of $\mu_s(Q)$. We note that $\mu_s(Q)$ is the quiver obtained by successively mutating $Q$ at each vertex of $s$ under the usual process of quiver mutation.

\begin{definition}
Let $\J=(\{M_i\}_{i\in V},Q,a,\{\sigma_i\}_{i\in V})$ be a quantum monoidal seed with automorphism. A mutation of $\J$ in direction $s$ is a quantum monoidal seed with automorphism of the form
\[
 \mu_s(\J)=(\{M_i\}_{i\in V\setminus s} \cup \{M_i'\}_{i\in s},\mu_s(Q),a,\{\sigma_i\}_{i\in V\setminus s} \cup \{\sigma_i'\}_{i\in s})
\]
where for every $k\in s$, the module $M_k'$ and the automorphism $\sigma_k'\map{a^*M_k'}{M_{ak}'}$ are such that there is the following commutative diagram with exact rows:
\begin{equation}\label{eq:mutation}
  \begin{CD}
   0@>>>a^*(q\sodot_{i\to k}M_i)@>>>a^*(q^{\tilde{\Lambda}(M_k,M_k')}M_k\circ M_k') @>>> a^*\sodot_{k\to i}M_i@>>>0\\
  @. @VV\odot\sigma_i V   @VV\sigma_k\circ\sigma'_k V @VV \odot \sigma_i V\\
   0@>>>q\sodot_{i\to k}M_{ai}@>>>q^{\tilde{\Lambda}(M_k,M_k')}M_{ak}\circ M_{ak}'@>>> \sodot_{k\to i}M_{ai} @>>> 0
  \end{CD}
 \end{equation} where the products are always taken over edges in $Q$.
\end{definition}

Here the integer $\tilde{\Lambda}(M,M')$ is defined by $2\tilde{\La}(M,M')=\nu\cdot\nu'+\deg r_{M,M'}$, where $M$ and $M'$ are modules for $R(\nu)$ and $R(\nu')$. An alternative a postiori formula for these degree shifts that do not depend on degrees of $R$-matrices is given in \cite[Eq (6.3)]{kkkocombined}.
 
It is clear that if a mutation in direction $s$ exists, then it is unique. However the existence of such a mutation is highly nonobvious.

% $(Q,\{M_i,\sigma_i\}_{i\in V(Q)},\{\sigma_i'\}_{i\in V(Q)_{\rm ex}})$.
% A triple $(\{M_i\},Q,\sigma)$ is admissible if
% \begin{enumerate}
%  \item $\{M_i\}_{i\in V(Q)}$ is a commuting family of real self-dual simple objects.
%  \item $Q$ is an $a$-equivariant quiver with no edges between two vertices in the same $a$-orbit.
%  \item $\sigma_i\map{a^*M_i}{M_{ai}}$ is an isomorphism such that
%  \item For all $k\in J_{\rm ex}$, there exists a real self-dual module $M_k'$ and an isomorphism $\sigma'\map{a^*M_k'}{M_{ak}'}$ such that we have a the following commutative diagram with horizontal exact sequences:
% \[
%   \begin{CD}
%    0@>>>a^*(q\sodot_{i\to k}M_i)@>>>a^*q^{\tilde{\Lambda}(M_k,M_k')}M_k' @>>> a^*\sodot_{k\to i}M_i@>>>0\\
%   @. @VV\sodot\sigma_i V   @VV\sigma'_k V @VV \sodot \sigma_i V\\
%    0@>>>q\sodot_{i\to k}M_{ai}@>>>q^{\tilde{\Lambda}(M_k,M_k')}M_{ak}'@>>> \sodot_{k\to i}M_{ai} @>>> 0
%   \end{CD}
%  \]
%  and $M_k'$ commutes with $M_i$ for all $i\neq k$ and with $M_j'$ for $j\neq k$ in the same $a$-orbit as $k$.
% \end{enumerate}

%The following lemma is straightforward:
%
%\begin{lemma}
% Let $\J=(\{M_i\}_{i\in V},Q,a,\{\sigma_i\}_{i\in V})$ be a quantum monoidal seed with automorphism. Choose $i\in V$ and let $d$ be the smallest positive integer such that $a^di=i$. Then
% \[
%  (M_i\odot M_{ai}\odot\cdots \odot M_{a^{d-1}i},\sigma_i\odot\sigma_{ai}\odot\cdots\odot \sigma_{a^{d-1}i})
% \]
%is an object of $\C$.
%\end{lemma}

\begin{definition}
 Let $\J=(\{M_i\}_{i\in V},Q,a,\{\sigma_i\}_{i\in V})$ be a quantum monoidal seed with automorphism. It is said to be compatible with $\B^*$ if
 \[
  [ (M_i\odot M_{ai}\odot\cdots \odot M_{a^{d-1}i},\sigma_i\odot\sigma_{ai}\odot\cdots\odot \sigma_{a^{d-1}i})]\in \B^*
 \]
for all $i\in V$, where $d$ is the minimal integer such that $a^di=i$.
\end{definition} 

Note that by \cite[Prop 3.2.5]{kkkocombined},
$( M_i\odot M_{ai}\odot\cdots \odot M_{a^{d-1}i},\sigma_i\odot\sigma_{ai}\odot\cdots\odot \sigma_{a^{d-1}i})$ is a self-dual simple element of $\mathcal{C}$, so its class is a priori a root of unity times an element of $\B^\ast$.

%%%%%%%%%%%%%%%%%%%%%%%%%%%%%%%%%%%%%%%%%%%%%%%%%%%%%%%%%%%%%%%%%%%%%%
\section{The initial seed}\label{sec:initseed}
%%%%%%%%%%%%%%%%%%%%%%%%%%%%%%%%%%%%%%%%%%%%%%%%%%%%%%%%%%%%%%%%%

Recall the setup from \S 5. Thus we have a KLR algebra associated with vertex set $I$, that has a finite order automorphism $a$. Associated to $I$ and $a$, we constructed a Cartan datum $(J,\cdot)$ to which we attach a symmetrisable Kac-Moody algebra $\g$. Let $W$ be the Weyl group of $\g$. Let $W'$ be the Weyl group of the symmetric Cartan datum associated to $I$. Then there is a canonical embedding of $W$ in $W'$ sending each $s_j$ to $\prod_{i\in j} s_i$. Throughout, we use standard notation such as $s_k$ to denote generators of the Weyl group, and $\w_k$ to denote fundamental weights, where the indices may come from either $I$ or $J$, and the location of the index tells us which root system we are considering.

Let $w\in W$. Let $w=s_{j_1}\ldots s_{j_n}$ be a reduced expression for $w$ in $W$. To this data, we will construct in this section a 
quantum monoidal seed with automorphism compatible with $\B^*$. This quantum monoidal seed with automorphism will serve as an initial seed in our categorification of $\A_q(\n(w))$.

From our reduced expression of $w$ in $W$, we obtain a reduced expression for $w\in W'$ by replacing each 
occurrence of $s_j$ by $\prod_{i\in j}s_i$. The order in which this product is written will not matter. Write 
$w=s_{i_1}\cdots s_{i_m}$ for the reduced expression thus obtained.

Let $Q_0$ be the quiver $Q(i_1,\ldots,i_m)$ as defined in \S \ref{subsec}. It is equipped with a decomposition of its vertex set $V$ into exchangeable and frozen vertices and has a canonical automorphism which we also denote by $a$ satisfying condition (2) in Definition \ref{def:qmswa}.

Let $v=(t,i_t)$ be a vertex of $Q$. Define
\[
 M_v=M(s_{i_1}\cdots s_{i_t}\w_{i_t},\w_{i_t}),
\]
the notation being as in Proposition \ref{mexists}.

It is shown in \cite{kkkocombined} that $M_i\circ M_j$ is simple for all $i,j\in V$. Indeed this collection of modules is the same as the monoidal seed constructed in that paper.

A collection of isomorphisms $\sigma_v\map{a^\ast M_v}{M_v}$ is constructed via the following lemma.

%For example, suppose we have the reduced expression $(1,2,3,1,2)$ in type $C_2^{(1)}$. The corresponding unfolded reduced expresion is $(1,2,0,3,1,2,0)$ in $A_3^{(1)}$. Continue example with characters of initial modules.
%
%In type $A_3$, suppose 132132 is the reduced word. Then the initial cluster variables are [1], [3], [132]+[312], [321], [123] and [2132]+[2312].
%
%\begin{theorem}
% (simply laced). The modules $M$ in an initial cluster lie in the cateogry $\C_w$.
%\end{theorem}
%
%\begin{proof}
%Actually see Lemma \ref{lem:initial}
%
% Lemma \ref{restrictgm} shows that $\Res(M) \cong L(\a_1)\otimes\cdots\otimes L(\a_n)$ (we need to know how to identify certain PBW elements as generalised minors, this is \cite[Proposition 7.4]{gls}. So by adjunction there is a nonzero map from $L_1\circ\cdots\circ L_n$ which identifies $M$ as $A(L_1,\ldots,L_n)$ as $M$ is irreducible. Hence $M$ is in $\C_w$.
%\end{proof}

\begin{lemma}\label{lem:initial}
 Let $w=s_{j_1}\cdots s_{j_N}$ be a reduced expression in $W$. For $i\in j_N$, let $M_i=M(w\w_i,\w_i)$. Then there exist isomorphisms $\sigma_i\map{a^*M_i}{M_{ai}}$ such that
 \begin{enumerate}
  \item  \[\sigma_{a^{n-1}i}\circ a^*\sigma_{a^{n-2}i} \circ \cdots \circ (a^*)^{n-1}\sigma_i\]
is the identity map on $(a^*)^nM_i = M_i = M_{a^n i}$, 
and 
\item \[
[(\sodot_{i\in j_N} M_i,\sodot_{i\in j_N} \sigma_i)] \in \B^*.
\]
 \end{enumerate}
\end{lemma}

\begin{proof}
 We proceed by induction on $N$, with the base case $N=0$ being trivial. So suppose $N\geq 1$. Let $N_i=N(s_{j_2}\cdots s_{j_N}\w_i,\w_i)$. By inductive hypothesis there exist isomorphisms $\tau_i\map{a^*N_i}{N_{ai}}$ such that
 \begin{equation}\label{taucomp}\tau_{a^{n-1}i}\circ a^*\tau_{a^{n-2}i} \circ \cdots \circ (a^*)^{n-1}\tau_i\end{equation}
is the identity map on $(a^*)^nN_i = N_i = N_{a^n i}$, 
and 
\[
[(\sodot_{i\in j_N} N_i,\sodot_{i\in j_N} \tau_i)] \in \B^*.
\]

Let $X_i=M(w \w_i, s_{j_2}\cdots s_{j_N}\w_i)$. Then $X_i$ is a simple module of $R(\nu)$ where $\nu$ is a sum of elements of $j_1$. Therefore $X_i$ is a circle product of one-dimensional simples, hence there is a canonical choice of isomorphism $x_i\map{a^*X_i}{X_{ai}}$.
By Theorem \ref{4.28}, $M_i\cong X_i\diamond N_i$. We define $\sigma_i$ to be the isomorphism induced from $x_i\circ \tau_i$. 
That condition (1) is satisfied follows from (\ref{taucomp}). To check (2), note that
 \[
  (\sodot_{i\in j_N} M_i,\sodot_{i\in j_N} \sigma_i)=(\tilde f_{j_1})^e (\sodot_{i\in j_N} N_i,\sodot_{i\in j_N} \tau_i)
 \]
for some integer $e$, where $\tilde f_j$ is the crystal operator from \cite[\S 9]{folding}. The basis $\B^*$ is defined in \cite{folding} in such a way that it is preserved by $\tilde f_j$, so this completes the proof.
\end{proof}

% \begin{lemma}\label{mscongmas}
%  For each vertex $s$ of $Q_0$, there exists an isomorphism $a^*M_s\cong M_{as}$.
% \end{lemma}
% 
% \begin{proof}
% Since both sides of this purported isomorphism are simple, it suffices to prove the identity in the Grothendieck group $a[M_s]=[M_{as}]$.
% 
% In any $a$-orbit of $Q_0$, the corresponding simple reflections in $W$ all commute and satisfy $s_k\w_l=\w_l$ for $k\neq l$. Therefore the corresponding modules $M_k$ are all of the form 
% \[
%  M_k=M(us_l \w_l, \w_l).
% \]
% where the Weyl group element $u$ is fixed by $a$ and only depends on the $a$-orbit of the vertex $k$. This shows the necessary identity $a[M_s]=[M_{as}]$.
% \end{proof}

In \cite{kkkocombined}, it is shown that for each vertex $k\in V_{\rm ex}$, there exists a real simple module $M_k'$ and a short exact sequence
\begin{equation}\label{msprimedefn}
 0\to q\sodot_{i\to k} M_i \to q^{\tilde{\La}(M_k,M_k')}M_k\circ M_k' \to \sodot_{k\to i} M_i \to 0.
\end{equation}

\begin{lemma}
 For each $s\in V_{\rm ex}$, there is an isomorphism $a^*M_s'\cong M_{as}'$.
\end{lemma}

\begin{proof}
 Since $a^*M_s'$ and $M'_{as}$ are simple, it suffices to prove an equality $a[M'_s]=[M'_{as}]$ in the Grothendieck group. Since the Grothendieck group $U_q(\mathfrak{n})$ is a domain, the class $[M'_s]$ is computed from the short exact sequence (\ref{msprimedefn}). By Lemma \ref{lem:initial}, there is an isomorphism between $a^*M_t$ and $M_{at}$ for all vertices $t$ of $Q_0$. This proves the desired equality $a[M'_s]=[M'_{as}]$.
\end{proof}

\begin{theorem}
 There exist choices of isomorphisms $\sigma_s\map{a^*M_s}{M_{as}}$ for $s\in V$ and $\sigma'_s\map{a^*M_s'}{M_{as}'}$ for $s\in V_{\rm ex}$ such that 
 \begin{enumerate} \label{cond:one}
  \item $$\sigma_{a^{n-1}i}\circ a^*\sigma_{a^{n-2}i} \circ \cdots \circ (a^*)^{n-1}\sigma_i$$
is the identity map on $(a^*)^nM_i = M_i = M_{a^n i}$, and 
$$\sigma'_{a^{n-1}i}\circ a^*\sigma'_{a^{n-2}i} \circ \cdots \circ (a^*)^{n-1}\sigma'_i$$
is the identity map on $(a^*)^nM'_i = M'_i = M'_{a^n i}$.
 \item for all $s\in V_{\rm ex}$, the commutative diagram (\ref{eq:mutation}) with exact rows holds. 
 \end{enumerate}
 \end{theorem}

 \begin{proof}
  
For each vertex $s$, choose an isomorphism $\tau_s\map{a^*M_s}{M_{as}}$ and for each exchangeable vertex $s$, choose an isomorphism $\tau'_s\map{a^*M'_s}{M'_{as}}$.

For each exchangeable $s$, consider the diagram
\begin{equation}\label{eq:cd1}
  \begin{CD}
   0@>>>a^*(q\sodot_{t\to s}M_t)@>>>a^*(q^{\tilde{\Lambda}(M_s,M_s')}M_s\circ M_s') @>>> a^*\sodot_{s\to t}M_t@>>>0\\
  @. @VV f V   @VV\tau_s\circ\tau'_s V @VV g V\\
   0@>>>q\sodot_{t\to s}M_{at}@>>>q^{\tilde{\Lambda}(M_s,M_s')}M_{as}\circ M_{as}'@>>> \sodot_{s\to t}M_{at} @>>> 0.
  \end{CD}
 \end{equation}
Here, $f$ and $g$ are isomorphisms canonically induced by the isomorphism $\tau_s\circ \tau'_s$ since $M_s\circ M'_s$ is uniserial.

Since $\odot_{t\to s} M_t$ and $\odot_{s\to t} M_t$ are irreducible and $k$ is algebraically closed, Schur's Lemma implies that there exist $\de_s,\e_s\in k^\times$ such that
\[
 f=\d_s\sodot_{t\to s} \tau_t,\qquad g=\e_s\sodot_{s\to t}\tau_t.
\]

Let $\sigma_s=\la_s\tau_s$ and $\sigma'_s=\mu_s\tau'_s$ for some constants $\la_s$ and $\mu_s$ which are to be determined. 

Note that (\ref{eq:mutation}) will be satisfied if and only if
\begin{equation}\label{eq:compatibility}
 \la_s\mu_s = \d_s \prod_{t\to s} \la_t \quad \mbox{and} \quad \la_s\mu_s = \e_s \prod_{s\to t} \la_t.
\end{equation}

For each $i\in I$, consider the leftmost vertex $s$ in $Q_0$ with second coordinate $i$. Make a choice of $\la_s$ for these classes such that the corresponding collection of $\sigma_s$'s satisfy the condition (\ref{cond:one}). This is possible by Lemma \ref{lem:initial}.
% a version of \cite[Lemma 3.4]{folding}. 
We now claim that once this choice is made, there is a unique solution to (\ref{eq:compatibility}) in the remaining variables $\la_s$ and $\mu_s$.

To see this, note first that it suffices first to solve the following system of equations in the variables $\la_t$:
\begin{equation}\label{lambdasystem}
 \d_s\prod_{t\to s}\la_t = \e_s \prod_{s\to t}\la_t.
\end{equation}

This system of equations determines the values of $\la_t$ recursively. To find the value of $\la_t$ in terms of those $\la_s$ with $s$ where $s$ has a smaller first coordinate, look at and rearrange the equation (\ref{lambdasystem}) where $s$ is the vertex immediately to the left of $t$.

This deals with the question of satisfying the second property in the theorem statement. For the first, we stack many copies of the diagram (\ref{eq:mutation}) on top of each other to obtain  
\[
\begin{CD}
   0@>>>(q\sodot_{i\to k}M_i)@>>>(q^{{\Lambda}(M_k,M_k')}M_k\circ M_k') @>>> \sodot_{k\to i}M_i@>>>0\\
  @. @VVf V   @VVg V @VV h V\\
   0@>>>q\sodot_{i\to k}M_{i}@>>>q^{{\Lambda}(M_k,M_k')}M_{k}\circ M_{k}'@>>> \sodot_{k\to i}M_{i} @>>> 0
  \end{CD} \]
  where $f$, $g$ and $h$ are all the appropriate compositions of circle products. Induct on the vertex again. By inductive hypothesis, $f=\id$. Since $\End(M_k\circ M_k')\cong k$ and the diagram commutes, $g=\id$.
  Therefore we get the property for the primed composition. Again since the diagram commutes, $h=\id$, so we get the desired property there too, completing the induction step.
 \end{proof}

\begin{definition}\label{def:initqms}
 The quadruple
\[
 \J_{\rm in}=(\{M_i\}_{i\in V},Q_0,a,\{\sigma_i\}_{i\in V})
\]
is the initial seed associated to the choice of reduced expression for $w$.
\end{definition}

\begin{proposition}
Suppose $v=(k,i_k)\in V$. Let $l_1,\ldots,l_r$ be the location of indices in the reduced expression $s_{i_1}\cdots s_{i_k}$ that are equal to $i_k$ (so in particular $l_r=k$).
 In the classification of irreducible modules in terms of a cuspidal decomposition from \cite{tingleywebster}, we have
\[
 M_v=\operatorname{hd}(L(\b_{l_r}),\ldots,L(\b_{l_1})).
\]
\end{proposition}

\begin{proof}
 For $0\leq m\leq r$, let $\mu_m=s_{i_1}s_{i_2}\cdots s_{i_{l_m}}\w_{i_k}$. We rewrite Theorem \ref{cuspgm} in the form $[L(\b_{l_m})]=D(\mu_m,\mu_{m-1})$. By Lemma \ref{restrictgm}, we then see that
\[
 \Res_{\mu_r-\mu_{r-1},\ldots,\mu_1-\mu_0}M_v \cong L(\b_{l_r})\otimes\cdots\otimes L(\b_{l_1}).
\]
This is enough to identify $M_v$ in the classification of irreducible modules in terms of semicuspidal decompositions.
\end{proof}

\begin{corollary}
 The modules $M_v$ in an initial cluster all lie in the cuspidal category $\C_{w}$.
\end{corollary}

%%%%%%%%%%%%%%%%%%%%%%%%%%%%%%%%%%%%%%%%%%%%%%%%%%%%%%%
\section{Cluster mutation}
%%%%%%%%%%%%%%%%%%%%%%%%%%%%%%%%%%%%%%%%%%%%%%%%%%%
% 
% In the previous section, we proved
% 
% \begin{proposition}\label{base}
%  Pick a reduced expression for $w$. The corresponding initial quantum monoidal seed with automorphism $\J_{\rm in}$ admits a mutation in each direction.
% \end{proposition}
% 
% \begin{proof}
%  Presumably this is \cite[Theorem 6.6]{kkko2} with a folding.
% %  
% %  We have short exact sequences
% %  \[
% %   0\to A\to X\to B\to 0
% %  \]
% %  and
% %  \[
% %   0\to a^*A \to a^* X \to a^* B\to 0
% %  \]
% %  The module $X$ is a circle product $X\cong P\circ Q$.
% %  So there is an isomorphism $a^*X\cong X$. Note it induces equivariant structures on its head and its socle.
% %  
% %  Therefore there is a formula in the Grothendieck group $b_Pb_Q=\zeta_1b_A+\zeta_2b_B$ where $\zeta_1$ and $\zeta_2$ are roots of unity. We want to use the positivity results for dual canonical basis multiplication to deduce that $\zeta_1$ and $\zeta_2$ are both 1.
% %  
% %  In \cite[Theorem 1.3(ii)]{clusterprime}, it is proved that cluster variables are irreducible, hence prime since a polynomial ring is a unique factorisation domain.
% \end{proof}

\begin{theorem}
Pick a reduced expression for $w$. Then
 the corresponding initial quantum monoidal seed with automorphism $\mathcal{J}_{\rm in}$ from Definition \ref{def:initqms} admits arbitrary mutations in all directions.
\end{theorem}

\begin{proof}
 We have to show that for any sequence $s_1,\ldots,s_t$ of elements in $K_{\rm ex}$, the mutation $\mu_{s_t}\cdots \mu_{s_1}(\mathcal{J}_{\rm in})$ exists.

We will achieve this via an induction on $t$. For $t\leq 1$, this was proved in the previous section. So now suppose that $t\geq 2$. Define
\[
 \mathcal{J}=(\{M_i\}_{i\in V},Q,a,\{\sigma_i\}_{i\in V})=\mu_{s_{t-2}}\cdots \mu_{s_1}(\mathcal{J}_{\rm in}).
\]
By inductive hypothesis, for each $s\in K_{\rm ex}$, the mutation
\[
 \mu_s(\mathcal{J})=(\{M_i\}_{i\in V\setminus s} \cup \{M_i'\}_{i\in s},\mu_s(Q),a,\{\sigma_i\}_{i\in V\setminus s} \cup \{\sigma_i'\}_{i\in s})
\] exists.
Let
\[
 \mu_{s_{t-1}}(\mathcal{J})=(\{N_i\}_{i\in V},\mu_{s_{t-1}}(Q),a,\{\tau_i\}_{i\in V}).
\]
As in \cite{kkkocombined}, we shall assume that there are only arrows in $Q$ from $s_{t-1}$ to $s_{t}$, the other case proceeding similarly. Let $y\in s_t$. Choose $x\in s_{t-1}$ Define modules
\[
 L_{x,y}=N_x^{\odot b_{xy}}
\]
and
\[
 A_{x,y}=\sodot_i M_i^{\odot \min(b_{ix}b_{xy},-b_{iy})}.
\]
A module $M''_y$ is then constructed in the proof of \cite[Theorem 7.1.3]{kkkocombined} such that
\[
 L_{xy}\diamond M_y'\cong M_y''\circ A_{xy}.
\]
Make a choice of isomorphism $f_y\map{L_{xy}\diamond M_y'}{M_y''\circ A_{xy}}$.

Now we have isomorphisms $\sigma'\map{a^*L_{x,y}}{L_{ax,ay}}$ and $\sigma\map{a^*A_{x,y}}{A_{ax,ay}}$ which are induced from $\mu_x(\mathcal{J})$ and $\mathcal{J}$. There is also an isomorphism $\sigma'\map{a^*M_y'}{M_{ay}'}$. Together they induce an isomorphism $\sigma'\map{a^*(L_{x,y}\diamond M_y')}{L_{ax,ay}\diamond M_{ay}'}$. Consider the diagram
\begin{equation}\label{taucd}
 \begin{CD}
   a^*(L_{x,y}\diamond M_y') @>\sigma'>>   L_{ax,ay}\diamond M_{ay}'\\
   @VVa^*f_y V   @VVf_{ay} V\\
   a^*M_y''\circ a^*A_{x,y} @>\tau>> M_{ay}''\circ A_{ax,ay}
 \end{CD}
\end{equation}
 There is a unique morphism $\tau$ as in the diagram above making it commute. Since $M''_y\circ A_{x,y}$ is irreducible, there is a unique $\tau_y\map{a^*M_y''}{M_{ay}''}$ such that $\tau=\tau_y\circ \sigma$. 
This $\tau_y$ is the morphism we seek.

The family of morphisms $\tau_y$ satisfy the condition (4) in Definition \ref{def:qmswa} if and only if the family $\tau$ do. Since the family $\sigma'$ satisfies this condition, this follows from the commutative diagram (\ref{taucd}).

% To show the desired mutation property, use (6.13) of \cite{unfolded}, pass it through the above diagram and remove the superfluous factor of $A$.

It remains to show the mutation property holds. 
Define
\begin{align*}
 B=\sodot_{b_{iy} > 0, \ b_{ix} >0} M_i^{\snconv b_{ix}b_{xy}} \nconv \sodot_{\substack{b_{iy}' >0, \; b_{iy} <0,\;b_{ix} >0}} M_i^{\snconv b_{iy}'}, \\
 P_y=\sodot_{b_{iy} >0, i\neq x} M_i^{\snconv b_{iy}}, \quad  Q_y=\sodot_{b_{iy}' <0, \ i\neq x} M_i^{\snconv-b_{yi}'}.
\end{align*}

We consider the following diagram, whose horizontal rows are the exact sequences constructed in the proof of \cite[Theorem 7.1.3]{kkkocombined} (two displayed equations above (7.13)). The integer $c$ is described explicitly in \cite{kkkocombined} and the vertical morphisms are composites of the canonical maps from $a^\ast X_y$ to $X_{ay}$ for $X\in\{B,P,A,M,L,M',Q\}$.
 
 \centerline{\xymatrix{
 0 \ar[r] & a^*(q(B_y\sodot P_y)\circ A_y)   \ar[r]\ar[d] 
& a^* (q^c M_y\circ (L_y\diamond M_y')) \ar[r]\ar[d] & a^* ((L_y\sodot Q_y)\circ A_y) \ar[r]\ar[d]
  & 0
 \\
   0 \ar[r] & q(B_{ay}\sodot P_{ay})\circ A_{ay} \ar[r] & q^c M_{ay}\circ (L_{ay}\diamond M_{ay}') \ar[r] & (L_{ay}\sodot Q_{ay})\circ A_{ay}  \ar[r] & 0.
}}
  This diagram is built from the short exact sequences for mutation between $M_y$ and $M_y'$, together with $R$-matrix constructions. Using (\ref{eq:mutation}) and $a^*r_{X,Y}=r_{a^*X,a^*Y}$, we see that this diagram commutes.
  
  We then combine this with (\ref{taucd}) to obtain a commutative diagram
  
  \centerline{\xymatrix{
 0 \ar[r] & a^*(q(B_y\sodot P_y)\circ A_y )  \ar[r]\ar[d] 
& a^* (q^c M_y\circ M_y''\circ A_y) \ar[r]\ar[d] & a^* ((L_y\sodot Q_y)\circ A_y) \ar[r]\ar[d]
  & 0
 \\
   0 \ar[r] & q(B_{ay}\sodot P_{ay})\circ A_{ay} \ar[r] & q^c M_{ay}\circ M_{ay}''\circ A_{ay} \ar[r] & (L_{ay}\sodot Q_{ay})\circ A_{ay}  \ar[r] & 0.
}}

The proof of \cite[Theorem 7.1.3]{kkkocombined} then shows that the following diagram has exact rows

 \centerline{\xymatrix{
 0 \ar[r] & a^*q(B_y\sodot P_y)    \ar[r]\ar[d] 
& a^* (q^c M_y\circ M_y'')  \ar[r]\ar[d] & a^* (L_y\sodot Q_y)  \ar[r]\ar[d]
  & 0
 \\
   0 \ar[r] & q(B_{ay}\sodot P_{ay})  \ar[r] & q^c M_{ay}\circ M_{ay}''  \ar[r] & (L_{ay}\sodot Q_{ay})   \ar[r] & 0.
}}

We need to show that this diagram commutes to show that $\mu_{s_t}\cdots \mu_{s_1}(\mathcal{I}_{\rm in})$ exists. It commutes because the diagram above it is obtained by applying $-\circ A$, and that diagram commutes.
  \end{proof}

\section{Decategorification}
%%%%%%%%%%%%%%%%%%%%%%%%%%%%%%%%%%%%%%%%%%%%%%%%%%%%%%%%%%%%%%%%%%%%%%%%

Now we have established the existence of a quantum monoidal seed with automorphism, that has successive mutations in all directions. There was a choice made in the construction of the initial seed, but Theorem \ref{independence} below shows that this choice does not matter for the categorified quantum cluster structure. We now investigate what implications this categorical structure has for a cluster algebra structure on $\A_q(\n(w))$.

Suppose $\J=(\{M_i\}_{i\in V},Q,a,\{\sigma_i\}_{i\in V})$ is a reachable quantum monoidal seed with automorphism in $\C_w$. From $\J$, we now give a recipe for constructing a cluster in $\A_q(\n(w))$.

For each $a$-orbit $s$ of $V(Q)$, we define the cluster variable
\begin{equation}\label{xsdefn}
 x_s = \left[ \left( \sodot_{i\in s} M_i,\sodot_{i\in s}  \sigma_i \right) \right]
\end{equation}

The exchange matrix $B$ is constructed from $Q$ in the following fashion. Let $s$ and $t$ be two $a$-orbits on the vertex set $V$. Then $b_{ts}$ is equal to the number of arrows from one element of $t$ to all elements of $s$. Here if an arrow goes from $t$ to $s$ it counts as a positive arrow and if an arrow goes from $s$ to $t$ it counts as a negative arrow for the purposes of this count.

The matrix $\La=(\la_{st})$ is uniquely determined by the equations $x_s x_t=q^{\la_{st}}x_tx_s$, using Lemma \ref{lem:qcommute} below.

Denote the collection $(\{x_s\},B,\La)$ by $X(\J)$.

\begin{lemma}\label{lem:qcommute}
 The elements $x_s$ defined by (\ref{xsdefn}) all q-commute.
\end{lemma}

\begin{proof}
By \cite[Proposition 3.2.5]{kkkocombined}, $x_sx_t$ and $x_tx_s$ are both classes of simple modules. These must be the same simple, hence they are equal up to a power of $q$.
\end{proof}

\begin{lemma}\label{16.5}
Suppose that $\J$ is a quantum monoidal seed with automorphism that admits a mutation in the direction $s$. Then the clusters $X(\J)$ and $X(\mu_s(\J))$ are related by a quantum cluster mutation as in (\ref{12.n}).
\end{lemma}

\begin{proof}Let $k\in s$ and let $d$ be the smallest positive integer such that $a^dk=k$. Define
\[
 \M_s=(M_k\odot M_{ak}\odot\cdots \odot M_{a^{d-1}k},\sigma_k\odot\sigma_{ak}\odot\cdots\odot\sigma_{a^{d-1}k}).
\] and $\M_s'$ and $\M_i$ for $i\neq s$ similarly.
Let $Z=\M_s\odot \M_s'$. 
%  Consider
%  \[
%   q^\La (\sodot_{k\in s} M_k \circ \sodot_{k\in S} M_k',\sodot_{k\in s} \sigma_k \circ \sodot_{k\in s}\sigma_k').
%  \]
By (\ref{eq:mutation}), there is a canonical map from $Z$ to $\odot_{k\to i}\M_i$ and a canonical map from $\odot_{i\to k}\M_i$ to $Z$.
Therefore
\[
 x_s x_s' = \prod_{k\to i} x_i + q^{e_s}\prod_{i\to k} x_i +c
\] 
where $c$ corresponds to the other terms in the Jordan-Holder filtration of $\M_s\odot \M_s'$.

 Let us now forget the diagram automorphism and consider the underlying KLR module only.
The composition factors of each $M_j\circ M_j'$ for each $j\in s$ are known by (\ref{eq:mutation}). 
Therefore the only simple subquotients $L$ of the underlying KLR module of $\M_s\circ\M_s'$ satisfying $a^*L\cong L$ are the submodule 
$\odot_{k\to i}\M_i$ and the quotient module
 $\odot_{i \to k}\M_i$ identified above. 
Returning to the folded situation, this implies that apart from the simple submodule and quotient explicitly identified, all of the simple composition factors of $\M_s\circ \M_s'$ are traceless by \cite[Theorem 3.6]{folding}. Hence $c=0$ as required.
\end{proof}

\begin{lemma}\label{initialminors}
% Let $w=s_{j_1}\cdots s_{j_N}$ be a reduced expression for $w\in W$.
  Let  $(\{M_i\}_{i\in V},Q_0,a,\{\sigma_i\}_{i\in V})$ be the initial quantum monoidal seed with automorphism arising from a reduced expression  $w=s_{j_1}\cdots s_{j_N}$ of some $w\in W$. Let $1\leq k\leq N$. To $k$ corresponds an $a$-orbit in $Q_0$. Let $s,as,\ldots,a^{d-1}s$ be an enumeration of this orbit. Let $v=s_{j_1}\ldots s_{j_k}$ Then
 \[
  [(\sodot_{i\in j_k} M_{a^is},\sodot_{i\in j_k} \sigma_{a^is})]=D(v\w_{j_k},\w_{j_k}).
 \]
\end{lemma}

\begin{proof}
We will first prove this identity holds up to a factor of a root of unity, then eliminate the root of unity ambiguity in the proof of Theorem \ref{17.3} below.

The module $\odot_{i\in j_k} M_{a^is}$ is simple. Therefore the left hand side of our desired identity lies in $\B^\ast$ up to a root of unity.

The module $V(\w_{j_k})$ is the quotient of $U_q(\mathfrak{n})$ by the left ideal generated by $\theta_{j_k}^2$ and $\theta_l$ for $l\neq j_k$. By Theorem \ref{basisofvla} and the fact that the $v\w_{j_k}$-weight space of $V(\w_{j_k})$ is one-dimensional, the generalised minor $D(v\w_{j_k},\w_{j_k})\in \B^\ast$ is characterised up to scalar by the fact that $r_{\nu-2j_k,2j_k}D(v\w_{j_k},\w_{j_k})=0$ and $r_{\nu-l,l}=0$ for $l\neq j_k$, where $\nu=\w_{j_k}-v\w_{j_k}$.

Therefore to complete the proof up to a root of unity, it suffices to show that
\[
\Res_{\nu-2j_k,2j_k}(\sodot_{i\in j_k} M_{a^is},\sodot_{i\in j_k}\sigma_{a^is})=0\ \ \mbox{and}\ \Res_{\nu-l,l}(\sodot_{i\in j_k} M_{a^is},\sodot_{i\in j_k}\sigma_{a^is})=0\ \mbox{for}\  l\neq j_k.\]

These restrictions can be computed at the unfolded level. It is already known that $[\odot_{i\in j_k} M_{a^is}]= \prod_{i\in j_k} D(v\w_i,\w_i)$, which is enough information to make this conclusion.

This completes the proof up to a root of unity, which is enough for its application in the proof of Theorem \ref{17.3} below. The conclusion of Theorem \ref{17.3} below implies that the root of unity ambiguity can be removed, completing the proof.
\end{proof}
%
%\begin{proof}
% Since both sides of the identity lie in $\B^*$, it suffices to check this identity at $q=1$. By folding, it suffices to prove the identity
% \[
%  \prod_{i\in j_k} D(v\w_i,\w_i) = D(w\sum_{i\in j_k} \w_i,\sum_{i\in j_k}\w_i)
% \]
%in $\Z[N]$.
%
%This follows from a similar argument to Lemma \ref{squaregm}. It is required to show that the $(w\sum_{i\in j_k}\w_i)$-weight space of $\otimes_{i\in j_k}V(\w_i)$ is one-dimensional. But the dimension of a weight space is invariant under the Weyl group action, so it suffices to show that the $(\sum_{i\in j_k}\w_i)$-weight space is one-dimensional. This is clear as it is the highest weight.
%\end{proof}

%
%Let $\J$ be a quantum monoidal seed with automorphism. Let $s\in V$ and let $d$ be the smallest positive integer such that $a^ds=s$. Define
%\[
% \M_s=(M_s\odot M_{as}\odot\cdots \odot M_{a^{d-1}s},\sigma_s\odot\sigma_{as}\odot\cdots\odot\sigma_{a^{d-1}s}).
%\]

\begin{theorem}\label{17.3}
 Let $(\{M_i\}_{i\in V},Q_0,a,\{\sigma_i\}_{i\in V})$ be the initial seed as above. Then %$\sigma_s\map{a^*M_s}{M_{as}}$ for $s\in Q_0$ and $\sigma'_s\map{a^*M_s'}{M_{as}'}$ for $s\in Q_0^{\text{ex}}$ such that 
 \[
  [(M_s\odot M_{as}\odot\cdots\odot M_{a^{d-1}s}, \sigma_s \odot \sigma_{as}\odot\cdots\odot \sigma_{a^{d-1}s})]\in \B^*
 \]
 and
 \[
  [(M'_s\odot M'_{as}\odot\cdots\odot M'_{a^{d-1}s}, \sigma'_s \odot \sigma'_{as}\odot\cdots\odot \sigma'_{a^{d-1}s})]\in\B^*,
 \] where $d$ is the minimal integer such that $a^ds=s$.
 \end{theorem}

\begin{proof}
Let $r\in V$. If it exists, let $s$ be the vertex immediately to the left of $r$ in the same row of $Q_0$. We will assume as our inductive assumption that the classes of the objects
 \[
  (M_t\odot M_{at}\odot\cdots\odot M_{a^{d-1}t}, \sigma_t \odot \sigma_{at}\odot\cdots\odot \sigma_{a^{d-1}t})
 \]
lie in $\B^*$ for all $t\in V$ with smaller first coordinate than $r$. From this assumption, we will conclude that
 \begin{equation}\label{indhyp}
  [(M_s\odot M_{ar}\odot\cdots\odot M_{a^{d-1}r}, \sigma_r \odot \sigma_{ar}\odot\cdots\odot \sigma_{a^{d-1}r})]\in \B^*
 \end{equation}
 and 
 \begin{equation}\label{indhypprime}
  [(M'_s\odot M'_{as}\odot\cdots\odot M'_{a^{d-1}s}, \sigma'_s \odot \sigma'_{as}\odot\cdots\odot \sigma'_{a^{d-1}s})]\in \B^*
 \end{equation}
which suffices to prove our theorem.

If $r$ is the first vertex in its row, then we are done by Lemma \ref{lem:initial}. Suppose this is not the case and let $s$ be the vertex immediately to the left of $r$ in its row. Let $d$ be the smallest positive integer such that $a^dr=r$. Taking the circle product of $d$ copies of (\ref{eq:mutation}), we see that
\[
K:=q^d \left(\sodot_{t\to s} \sodot_{i=0}^{d-1}M_{a^{i}t},\sodot_{t \to s} \sodot_{i=0}^{d-1} \sigma_{a^it}\right)
\]
is the socle of 
\[
 Z:=((q^{d\tilde{\Lambda}(M_s,M_s')}M_s\odot\cdots\odot M_{a^{d-1}s}) \circ (M_s'\odot \cdots\odot M'_{a^{d-1}s}),(\sigma_s\odot\cdots\odot \sigma_{a^{d-1}s})\circ (\sigma'_s\odot\cdots\odot \sigma'_{a^{d-1}s}))
\]
and 
\[
Q:=(\sodot_{s\to t}\sodot_{i=0}^{d-1}M_{a^it},\sodot_{s\to t} \sodot_{i=0}^{d-1} \sigma_{a^it}) 
\]
is the head of $Z$.

Since each $M_k\circ M_k'$ is of length two, all simple subquotients of $Z$ apart from $K$ and $Q$ are traceless, so $[Z]=[K]+[Q]$ in the Grothendieck group.

By inductive hypothesis, we know $[Q]\in q^\Z\B^\ast$. There exists some root of unity $\zeta$ such that $\zeta [K]\in q^\Z \B^\ast$ and $\zeta=1$ if and only if (\ref{indhyp}) holds. The class of $Z$ is, up to a power of $q$, the product of $x=[(M_s\odot M_{as}\odot\cdots\odot M_{a^{d-1}s}, \sigma_s \odot \sigma_{as}\odot\cdots\odot \sigma_{a^{d-1}s})]$ and $x'=([M'_s\odot M'_{as}\odot\cdots\odot M'_{a^{d-1}s}, \sigma'_s \odot \sigma'_{as}\odot\cdots\odot \sigma'_{a^{d-1}s})]$. We know by inductive hypothesis that $x\in \B^\ast$. We know that $\zeta'x'\in\B^\ast$ for some root of unity $\zeta'$, which is 1 if and only if (\ref{indhypprime}) holds.

We now specialise to $q=1$. We get an equation of the form
\[
x (\zeta'x')=\zeta'\zeta\inv a+\zeta' b
\]
where and $a$ and $b$ are, by Lemma \ref{initialminors}, products of generalised minors. By \cite{kmcluster}, these generalised minors are cluster variables in an initial cluster of a cluster algebra structure on $\Cx[N(w)]$, so $x$ divides $a+b$.

In \cite[Theorem 1.3]{clusterprime}, it is shown that cluster variables are irreducible, which implies they are prime in our case as $\Cx[N(w)]$, being a polynomial ring, is a unique factorisation domain. If $\zeta\neq 1$, then we can use the two divisibility results to show that $x$ divides both $a$ and $b$ in $\mathbb{C}[N(w)]$, contradicting this primality result.

Therefore $\zeta=1$. Lemma \ref{product} now implies that $\zeta'=1$ which completes the proof.
\end{proof}

\begin{corollary}\label{initialcomp}
 The initial quantum monoidal seed $ \J_{\rm in}=(\{M_i\}_{i\in V},Q_0,a,\{\sigma_i\}_{i\in V})$ is compatible with $\B^*$.
\end{corollary}

% It is known that the initial seed lies in $\B^*$.
% 
% Exchange matrix:
% 
% Make the same WLOG choice as in \cite{unfolded}
% 
% Consider
% \[
%  \sodot_{k\in s}M_k \odot\sodot_{k\in s}M_k'.
% \]
% 
% There is a canonical (?!)  morphism to blah and one from blah. As long as each occurence of $M_k$ occurs to the left of $M_k'$, the order doesn't matter in a canonical fashion.
% 
% 
% By compatibility between $\sigma_i$'s, we get
% 
% \[
%  x_s x_s'=a+b+c
% \]
% where $a$ is the socle, $b$ is the head and $c$ corresponds to the other composition factors.
% 
% In unfolded land, the other composition factors are all known. They're all ``mixed'', there are no $L$ such that $a^*L\cong L$ amongst them. Therefore they will all be traceless. So in $K_0$, we have
% \[
% q^* x_s x_s'=q^*a+b
% \]
% and now we can prove the theorem about lying in $\B^*$.

\begin{theorem}\label{reachable}
 All reachable quantum monoidal seeds with automorphism are compatible with $\B^*$.
\end{theorem}

\begin{proof}
For an initial seed, this is Corollary \ref{initialcomp}. 
We now proceed by induction. So suppose that $\J=(\{M_i\}_{i\in V},Q,a,\{\sigma_i\}_{i\in V})$ is a quantum monoidal seed with automorphism compatible with $\B^\ast$ and $s$ be an $a$-orbit in $V_{\ex}$. We will show that $\mu_s(\J)$ is also compatible with $\B^\ast$.

Let $i\in s$ and $Y'_s=
[M'_i\odot M'_{ai}\odot\cdots \odot M_{a^{d-1}i}',\sigma'_i\odot\sigma'_{ai}\odot\cdots\odot \sigma_{a^{d-1}i}']$.
We have to show that $Y_s'\in \B^\ast$.
 By Lemma \ref{16.5}, we have $Y_sY_s'=x+y$ where $Y_s,x,y\in q^\Z\B^\ast$ by the inductive hypothesis. Furthermore $Y_s'$ is a root of unity times an element of $\B^\ast$, since it is the class of a self-dual simple object. Lemma \ref{product} says that a product of two elements of $\B^*$ has at least one coefficient a power of $q$ when expanded in the basis $\B^*$. Therefore the root of unity must be one, completing the proof.
\end{proof}

%The above theorem implies that every cluster monomial lies in $\B^\ast$.
%
%Recall the dual PBW basis vectors $E_k$ from \S \ref{qur}.
%
%\begin{theorem}
% Let $E_k^*$ be a dual PBW basis vector. Then there exists a sequence of mutations from the initial cluster to one in which $E_k^*$ is a cluster variable.
%\end{theorem}
%
%\begin{proof}
% Reduce to the unfolded case. (at this point perhaps we can give a citation, but we strive for some semblance of self-containment and so provide a proof). We need to study the magic mutation sequence of \cite[\S 13]{glsq1}. The key observation is that the relevant row is the only place in which modules of the form $A(X,L_k)$ can occur, and after each such mutation, the degree at the mutated vertex is $\alpha_k$ plus the degree at the vertex to the left (my convention, right in \cite{glsq1}). So at the end we have a $R(\alpha_k)$-module of the form $A(X,L_k)$, which is thus $L_k$.
%\end{proof}

At this point in time, our cluster categorification depends on the choice of a reduced expression for $w$. This is not the case because of the following theorem.

\begin{theorem}\label{independence}
 The initial clusters for each choice of reduced decomposition are all reachable from each other.
\end{theorem}

\begin{proof}
Note that if it exists, there is a unique up to isomorphism quantum monoidal seed with automorphism compatible with $\B^*$ categorifying any particular cluster. 
Therefore this theorem follows from its decategorified version, i.e. that any two initial clusters coming from different reduced decompositions are reachable from each other in the cluster algebra. This is a standard fact about the cluster algebra structure on $\mathbb{C}[N(w)]$ - if two reduced decompositions are related by a braid relation of the form $s_is_js_i=s_js_is_j$, then the corresponding initial clusters are a single cluster mutation apart from each other.
\end{proof}

\section{Conclusion}
%%%%%%%%%%%%%%%%%%%%%%%%%%%%%%%%%%%%%%%%%%%%%%%%%%%%%%%%%%%%%%%%%

The following theorem is \cite[Theorem 10.1]{my}, generalised in \cite{gy2}. The proof of Goodearl and Yakimov is purely non-commutative ring theoretic.

\begin{theorem}
The $\Z[q,q\inv]$-algebra
 $A_q(\n (w))$ has the structure of a quantum cluster algebra. An initial cluster is given by the collection of generalised minors in (\ref{eq:defyt}), while the exchange matrix for the initial cluster is given by the construction in \S \ref{sec:initseed}.
\end{theorem}
We have categorified this theorem with the results of this paper.

Now recall that $\B^\ast$ denotes the dual $p$-canonical basis. As a corollary of our categorification results, we obtain:

\begin{theorem}\label{main}
Every cluster monomial in $\A_q(\n(w))$ lies in $\B^\ast$.
\end{theorem}

\begin{proof}
For the cluster variables, this is a consequence of Theorem \ref{reachable}. That the monomials in the cluster variables lie in $\B^\ast$ then follows up to a root of unity by \cite[Prop 3.2.5]{kkkocombined}. That the root of unity is 1 is Lemma \ref{product}.
\end{proof}

\def\cprime{$'$}

\end{document}